\documentclass[11pt]{amsart}
\usepackage{graphicx}
\usepackage[USenglish,english]{babel}
\usepackage[T1]{fontenc}
\usepackage[latin1]{inputenc}
\usepackage{amsfonts}
\usepackage{amssymb}
\usepackage{amsthm}
\usepackage{amsmath}
\usepackage{fancyhdr}
\usepackage[dvipsnames]{xcolor}

\usepackage{tikz}
\usepackage{tikz-cd}
\usetikzlibrary{decorations.pathreplacing}
\usetikzlibrary{shapes.misc}
\tikzset{cross/.style={cross out, draw=black, fill=none, minimum size=2*(#1-\pgflinewidth), inner sep=0pt, outer sep=0pt}, cross/.default={2pt}}
\usepackage{float}
\usepackage{enumerate}
\usepackage{accents}
\usepackage{mathtools}
\usepackage{mathrsfs}
\usepackage{comment}
\usepackage{afterpage}
\usepackage{bbm}
\usepackage{dsfont}
\usepackage{stmaryrd}
\usepackage{mleftright}
\usepackage{subfig}
\usepackage{faktor}
\usepackage{graphicx}
\usepackage{marvosym}
\usepackage[all]{xy}
\usepackage[linktocpage]{hyperref}

\newtheorem{theorem}{Theorem}
\numberwithin{theorem}{section}

\newtheorem{lemma}[theorem]{Lemma}

\newtheorem{proposition}[theorem]{Proposition}

\newtheorem{conjecture}[theorem]{Conjecture}

\newtheorem{corollary}[theorem]{Corollary}
\newtheorem{question}[theorem]{Question}
\newtheorem{thmintro}{Theorem}

\newtheorem{lemmaintro}[thmintro]{Lemma}

\newtheorem{queintro}[thmintro]{Question}
\newtheorem{conjectureintro}[thmintro]{Conjecture}

\theoremstyle{remark}
\newtheorem{remark}[theorem]{Remark}

\theoremstyle{definition}
\newtheorem{definition}[theorem]{Definition}
\newtheorem{example}[theorem]{Example}

\newcounter{ccomments}
\newcommand{\carolyn}[1]{\textbf{\color{red}(C\arabic{ccomments})} \marginpar{\scriptsize\raggedright\textbf{\color{red}(C\arabic{ccomments})Carolyn: }#1}
	\addtocounter{ccomments}{1}}	
	
\newcounter{mcomments}

\newcommand{\R}{\mathbb{R}}

\newcommand{\AH}{\mathcal{AH}}
\newcommand{\Hyp}{\mathcal{HYP}}

\newcommand{\Vis}{\mathcal{VIS}}
\newcommand{\mc}{\mathcal}
\newcommand{\Ima}{\text{Im}}

\newcommand{\diam}{\textrm{diam}}

\DeclareMathOperator{\CAT}{CAT(0)}

\title[Hyperbolic projections and topological invariance]{Hyperbolic projections and topological invariance of sublinearly Morse boundaries}

\author{Carolyn Abbott}
\address{Brandeis University, USA}
\email{carolynabbott@brandeis.edu}

\author{Merlin Incerti-Medici}
\address{Karlsruher Institut f\"ur Technologie, Germany}
\email{merlin.medici@gmail.com}

\begin{document}

\begin{abstract}
We show that the sublinearly Morse boundary of a $\CAT$ cubical group with a factor system is well-defined up to homeomorphism with respect to the visual topology.
The key tool used in the proof is a new topology on sublinearly Morse boundaries that is induced by group actions on hyperbolic spaces that are sufficiently nice, 
 for example, largest acylindrical actions. Using the same techniques, we obtain a explicit description of this new topology on the sublinearly Morse boundary of any hierarchically hyperbolic group in terms of medians. Finally, we explicitly describe the sublinear Morse boundaries of graph manifolds using their actions on Bass-Serre trees.

\end{abstract}

\maketitle





\section{Introduction} \label{sec:Introduction}





Boundaries at infinity are a common and useful tool in the study of large-scale geometric properties of groups. Whenever a group acts geometrically on a metric space with suitable negative curvature properties, the boundary carries rich topological, dynamic, and metric structures that can be used to study the group. A particularly fruitful instance of this is Gromov-hyperbolic metric spaces and their visual boundaries. Crucially, any quasi-isometry between geodesic hyperbolic metric spaces induces a homeomorphism between their visual boundaries, which allows one to define the visual boundary of a hyperbolic \textit{group} (cf. \cite{Gromov-HypGps}).  

The situation becomes more complicated when  considering a group acting geometrically on non-hyperbolic spaces, even if the spaces still have some features of negative curvature. In particular, 
Croke and Kleiner provided an example of a group acting geometrically on two different, quasi-isometric $\CAT$ spaces which have non-homeomorphic visual boundaries \cite{CrokeKleiner00}, demonstrating that the visual boundary of a $\CAT$ group is not well-defined.  

Over the last two decades,  various methods have been developed to circumvent this issue, with the goal of applying tools from asymptotic geometry to $\CAT$ spaces and even more general metric spaces and groups. One fruitful approach has been is to define a boundary using only geodesic rays that satisfy a Morse property, rather than  considering all geodesic rays, as for the visual boundary.  The  Morse property is meant to ensure that these rays have properties similar to geodesic rays in hyperbolic spaces.  This approach was initially used to define the Morse boundary of a $\CAT$ space 
 by Charney and Sultan \cite{CharneySultan}, and then expanded to include  general metric spaces by Cordes \cite{Cordes}.   Qing, Rafi, and Tiozzo finally generalised this notion even further by relaxing the required Morse property to define the \textit{sublinearly Morse boundary}  \cite{QingRafiTiozzo, QingRafiTiozzoII}, which is the focus of this article.  For every geodesic metric space $Y$ and every non-negative sublinear function $\kappa$, one can define the  $\kappa$-boundary $\partial_{\kappa} Y$, also called the sublinearly Morse boundary; see Section~\ref{sec:Preliminaries} for the precise definition. As a set, the $\kappa$-boundary is invariant under quasi-isometry, meaning that every quasi-isometry $\Psi \colon Y \rightarrow Y'$ canonically extends  to a bijection $\partial \Psi \colon \partial_{\kappa} Y \rightarrow \partial_{\kappa} Y'$. Several topologies have been developed with respect to which  $\partial \Psi$ is  a homeomorphism, turning the (sublinearly) Morse boundary into a topological quasi-isometry invariant \cite{CharneySultan, Cordes, CashenMackay, QingRafiTiozzo, QingRafiTiozzoII}.


Since $\kappa$-boundaries can be represented as equivalence classes of geodesic rays, for a proper $\CAT$ space $Y$ there is a canonical inclusion $\partial_{\kappa} Y \hookrightarrow \partial_{\infty} Y$, where $\partial_{\infty} Y$ denotes the visual boundary of $Y$. While the Croke-Kleiner example mentioned above shows that $\partial_{\infty} Y$ equipped with the visual topology is not invariant under quasi-isometry, one may wonder whether the restriction of the visual topology to the $\kappa$-boundary is a quasi-isometry invariant. The answer to this is negative in general, as was shown by Cashen, who constructed two quasi-isometric $\CAT$ spaces  whose Morse boundaries are not homeomorphic when equipped with the visual topology. \cite{Cashen}. However, the constructed spaces have trivial isometry group, and any attempts to recreate the same phenomenon for $\CAT$ spaces that admit geometric group actions have failed. This paper shows that as long as a group $G$ acts geometrically on two sufficiently nice $\CAT$ spaces, their $\kappa$-boundaries are, in fact, homeomorphic with respect to the visual topology.  

\begin{thmintro} \label{thmintro:Visualtopologyinvariance}
Let $G$ act geometrically and by cubical isomtries on   $\CAT$ cube complexes $Y$ and $Y'$ admiting  factor systems, and let $\kappa$ be any non-negative, increasing, sublinear function. The induced $G$-equivariant bijection $\partial \Psi \colon \partial_{\kappa} Y \rightarrow \partial_{\kappa} Y'$ is a homeomorphism with respect to the visual topology.
\end{thmintro}


See Definition~\ref{def:factorsystem} for the definition of a factor system. Due to the examples by Cashen,  our proof of Theorem~\ref{thmintro:Visualtopologyinvariance} necessarily takes a  rather roundabout approach. The majority  of this paper is the development of the necessary terminology and results that are of interest in a far more general context. 

The central notion is that 
of a {\it hyperbolic projection}, denoted $(G, Y, X)$, which consists of a group $G$ acting geometrically on a geodesic metric space $Y$ and  coboundedly  on a hyperbolic space $X$ that admits a $D$-quasi-ruling (see Definition \ref{def:unparametrisedquasirulers}).   
The terminology comes from the fact that given such a triple $(G, Y, X)$, there is a natural projection map $p\colon Y\rightarrow X$; see Section \ref{subsec:Acylindricalactions} for details.  
When $(G,Y,X)$ satisfies a condition we call \textit{$\kappa$-injectivity} (see Definition~\ref{def:kappainjectivity}), we are able to extend $p$ to an injective map $\partial p\colon \partial_\kappa Y\hookrightarrow \partial_\infty X$.  We denote the pullback of the visual topology on $\partial_\infty X$ under $\partial p$ by  $\mc T(G,Y,X)$; this is a topology on $\partial_\kappa Y$.

This topology is well-behaved in the following sense. Two spaces $X$ and $X'$ admitting actions of a group $G$ are equivalent, written $X\sim X'$,  if there is a coarsely $G$-equivariant quasi-isometry $X \rightarrow X'$.  

\begin{lemmaintro} \label{lem:IntroMainObservation}
Let $(G, Y, X)$ and $(G, Y', X')$ be two hyperbolic projections such that $X \sim X'$, and let $\Psi \colon Y \rightarrow Y'$ be a coarsely $G$-equivariant quasi-isometry. Then $(G, Y, X)$ is $\kappa$-injective if and only if $(G, Y', X')$ is $\kappa$-injective. Furthermore, if they are $\kappa$-injective, then $\partial \Psi$ is a homeomorphism with respect to the topologies $\mc T(G,Y,X)$ and $\mc T(G,Y',X')$.
\end{lemmaintro}

In particular, we obtain a well-defined topology on $\partial_{\kappa} G$ that is independent of the choice of geometric representation of $G$. We will often write $\mathcal{T}(G, [X])$ to denote this topology.

While we do not show that $\mc T(G,[X])$ is invariant under all quasi-isometries, the result above does extend to commensurable groups. Recall that two groups  are  {\it commensurable} if they have a common finite index subgroup.   If  $G$ and $ G'$ are commensurable and $(G, Y, X)$ and $(G', Y', X')$ are hyperbolic projections, then the actions of $G$ and $G'$ can be restricted to the common subgroup $H$.  Thus  we obtain hyperbolic projections $(H, Y, X)$ and $(H, Y', X')$, which are $\kappa$-injective if and only if the corresponding hyperbolic projections of $G$ and $G'$ are. We conclude the same statements as in Lemma \ref{lem:IntroMainObservation} for $(G, Y, X)$ and $(G', Y', X')$ and obtain that, if the actions of $H$ on $X$ and $X'$ are equivalent, then we have a homeomorphism $\partial_{\kappa}G \rightarrow \partial_{\kappa}G'$ with respect to $\mathcal{T}(G, [X])$ and $\mathcal{T}(G', [X'])$. In particular, we can think of the boundary $\partial_{\kappa} G$ with the topology $\mathcal{T}(G, [X])$ as a boundary of the commensurability class of $G$.

There are several interesting classes of groups and actions for which we can explicitly describe $\mc T(G,[X])$: cubulated groups, hierarchically hyperbolic groups, and graph manifold groups. These are all examples of groups that admit largest acylindrical actions, and we choose the equivalence classes of such actions as $[X]$.  In this case, we call the resulting topology  the {\it acylindrical topology}. While there is some overlap between these classes of groups, we  take advantage of some of their specific features and so  discuss each separately.


\subsection{Application to cubulated groups} \label{subsec:Applicationtocubulatedgroups}

We first consider the class of cubulations of groups that have factor systems (see Definition \ref{def:factorsystem}). 
Most known cubulations  admit a factor system \cite{HagenSusse20}, although Shepherd recently provided the first example of a $\CAT$ cube complex with a geometric group action that does not have a factor system \cite{Shepherd22}. If a cubulation of a group has a factor system, then there is a construction by Genevois in \cite{Genevois20b} that provides a $\kappa$-injective hyperbolic projection $(G,Y,Y_L)$, and $G\curvearrowright Y_L$  is a largest acylindrical action \cite{PetytSprianoZalloum22}.  See Section~\ref{subsec:CAT(0)spaces} for further discussion of these spaces and actions.

\begin{thmintro} \label{thmintro:Pullbacktopologyincubulatedcase}
Suppose $G$ acts geometrically and cubically on a $\CAT$ cube complex $Y$ that has a factor system. 
The topology $\mathcal{T}(G, Y, Y_L)$ coincides with the restriction of the visual topology on $\partial_{\infty} Y$ to $\partial_{\kappa} Y$.
\end{thmintro}




We briefly sketch how Theorem \ref{thmintro:Pullbacktopologyincubulatedcase} implies Theorem \ref{thmintro:Visualtopologyinvariance}. In short, the visual topology on the sublinearly Morse boundaries of different cubulations coincide, because it always coincides with the acylindrical topology.

\begin{proof}[Sketch of proof]
Since $Y$ and  $Y'$ each admit  a factor system,there are  hyperbolic spaces $Y_L$ and $Y'_L$ such that $(G, Y, Y_L)$ and $(G, Y', Y'_L)$ are  $\kappa$-injective hyperbolic projections and such that the action of $G$ on each is a largest acylindrical action. Since all largest acylindrical actions are equivalent, Lemma~\ref{lem:IntroMainObservation} shows that  the $G$-equivariant bijection $( \partial_{\kappa}Y, \mathcal{T}(G, Y, Y_L) ) \rightarrow ( \partial_{\kappa} Y', \mathcal{T}(G, Y', Y'_L) )$ is a homeomorphism. By Theorem~\ref{thmintro:Pullbacktopologyincubulatedcase}, these topologies coincide with the visual topology on $\partial_{\kappa} Y$ and $\partial_{\kappa}Y'$ respectively. 
\end{proof}

\subsection{Application to hierarchically hyperbolic groups} \label{subsec:applicationtoHHGs}

The second class of examples  we study are hierarchically hyperbolic groups, which all admit largest acylindrical actions \cite{AbbottBehrstockDurham21}.  We denote this largest action by $\mc C(S)$ (see Section \ref{subsec:ThecaseofHHGs}). 
Cubulated groups admiting factors systems are all examples of hierarchically hyperbolic groups.  In the absence of a cubical structure,  we instead take advantage of the coarse median structure $[\mu]$ on a hierarchically hyperbolic group to describe this topology.  
See Section \ref{sec:Topologyforgroupswithacoarsemedianstructure} for details on coarse medians. 

We  define the \textit{coarse median topology}, denoted $\mc T_{med}$, on $\partial_\kappa G$ and  show that this topology coincides with the acylindrical topology $\mc T(G,[\mc C(S)])$. The coarse median topology  can be described in terms of  convergence of median rays as follows. 
 Fixing a base point $o \in G$, every point in the sublinearly Morse boundary of $G$ can be represented by a median ray starting at $o$.  Given median rays $h_n, h$ in $G$ that start at $o$ and represent points $\xi_n, \xi \in \partial_{\kappa} G$, respectively, we say that $\xi_n \rightarrow \xi$ in $\mathcal{T}_{med}$ if and only if
\[ \liminf_{n \rightarrow \infty} \sup \{ d_G (o, \mu( o, h_n(s), h(t)) ) \mid s, t \in \mathbb{R}_{\geq 0} \} = \infty. \]
See Definition \ref{def:mediantopology} and Lemma \ref{lem:mediantopologybasis} for a precise definition and more details. 

\begin{thmintro} \label{thmintro:PullbacktopologyinHHGcase}
Let $G$ be a hierarchically hyperbolic group with a coarse median structure $[\mu]$, and let $\kappa$ be a sublinear function. The canonical bijection $( \partial_{\kappa} G, \mathcal{T}_{med} ) \rightarrow ( \partial_{\kappa} G, \mathcal{T}(G,[ \mathcal{C}(S)]) )$ is a homeomorphism. 
\end{thmintro}

\subsection{Application to graph manifolds} \label{subsec:Determiningthetopologyforgraphmanifolds}

Finally, we explicitly describe the $\kappa$-boundary with the acylindrical topology for fundamental groups of non-positively curved graph manifolds. Given a non-positively curved graph manifold, its fundamental group $G$ has a natural decomposition as a graph of groups. As such, $G$ acts on its Bass-Serre tree, which we denote by $T_{\infty}$.  It has been shown recently in \cite{HagenRussellSistoSpriano22} that this is a largest acylindrical action. 


\begin{thmintro} \label{thmintro:Graphmanifoldtopology}
Let $M$ be a closed, non-positively curved graph-manifold with fundamental group $G$ and universal cover $Y$.  The hyperbolic projection $(G, Y, T_{\infty})$ is  $\kappa$-injective,  and $( \partial_{\kappa} Y, \mathcal{T}(G, [T_{\infty}]))$ is a Cantor chain. 
\end{thmintro}

See Section~\ref{sec:MorseboundariesofRAAGsandgraphmanifolds} for the definition of a Cantor chain.
 The properties of graph manifolds used in the proof of Theorem \ref{thmintro:Graphmanifoldtopology} are not  specific to graph manifolds, and the proof is likely to generalise to a class of groups introduced in \cite{CrokeKleiner02} known as CK-admissible groups. These groups are also graphs of groups and thus admit an action on  their Bass-Serre tree $T_{\infty}$, which is a largest acylindrical action by \cite{HagenRussellSistoSpriano22}. We  make the following conjecture, which we discuss in Section~\ref{subsec:Comparisonwithvisualtopology}.

\begin{conjectureintro}
Let $G$ be a CK-admissible group that acts geometrically on a Hadamard space $Y$. The space $(\partial_\kappa Y,\mathcal{T}(G, [T_{\infty}]))$ is a Cantor chain and coincides with $\partial_{\kappa}Y$ equipped with the visual topology.
\end{conjectureintro}

\subsection{Open questions} \label{subsec:openquestions}

We end the introduction with a series of questions that arise from these results and our methods of proof.

Our results provide us with topologies that are well-defined on the $\kappa$-boundary of a finitely generated group $G$ (and its commensurability class), but we do not know if these topologies are generally invariant under quasi-isometry.  In other words, given two quasi-isometric groups $G$ and $G'$, we do not obtain any comparison between hyperbolic projections of $G$ and $G'$. A key obstacle to a generalisation of our results to such a setting is the existence of an  quasi-isometry $\Phi \colon X \rightarrow X'$ which commutes with the actions of $G$ and $G'$. In other words, one needs a suitable notion of an equivalence relation $X \sim X'$ when the acting group is no longer the same.

Moreover, if one aims to strengthen Theorem \ref{thmintro:Visualtopologyinvariance}, one would need to show that the largest acylindrical actions of  two quasi-isometric groups $G$ and  $G'$ that admit cubulations with factor systems are equivalent in this new sense. This leads us to the difficult question of quasi-isometry invariance of acylindrical hyperbolicity, especially invariance of a largest acylindrical action.

Alternatively, one may decide to consider not only cobounded actions $G \curvearrowright X$ on hyperbolic spaces, but cobounded quasi-actions, where quasi-actions are a suitable generalisation of isometric actions that can be pushed-forward under quasi-isometries. Instead of considering two quasi-isometric groups $G$ and  $G'$, one can  fix the group $G$ and turn every cobounded quasi-action of $G'$ into a cobounded quasi-action of $G$. Under this viewpoint, strengthening Theorem \ref{thmintro:Visualtopologyinvariance} to apply to quasi-isometry classes of groups likely boils down to the question whether the largest acylindrical action of a cubulable (with factor system) group is also a largest acylindrical quasi-action.

In light of these points, we raise the following questions.

\begin{queintro}
Can we make sense of largest acylindrical quasi-actions? For which cubulated groups is the largest acylindrical quasi-action equivalent to the largest acylindrical action? 
\end{queintro}


\begin{queintro}
Apart from largest acylindrical (quasi-)actions, are there other natural choices for an equivalence class of (quasi-)actions?
\end{queintro}

We finally discuss one other possible generalization of this work. The proof of Theorem \ref{thmintro:Pullbacktopologyincubulatedcase} relies on the work done in \cite{Genevois20b, Medici19, MediciZalloum21} about properties of hyperplanes in $\CAT$ cube complexes and their relation to $\kappa$-Morse geodesics. In the work of Petyt-Spriano-Zalloum \cite{PetytSprianoZalloum22}, these ideas are being generalised to $\CAT$ spaces. The hyperbolic space $Y_L$ that was introduced by Genevois is then replaced by the $L$-separation space introduced in \cite{PetytSprianoZalloum22}, which we also denote by $Y_L$ for the purpose of the next question. We ask whether their techniques can be used to obtain a `canonical' $\kappa$-injective hyperbolic projection for more $\CAT$ groups.

\begin{queintro}
Let $G$ act geometrically on two $\CAT$ spaces $Y$ and $Y'$ such that the metric spaces $Y_L$ and $Y'_L$ stabilize.
\begin{enumerate}
    \item Is $(G, Y, Y_L)$ a $\kappa$-injective hyperbolic projection for sufficiently large $L$?
    
    \item If so, are the actions of $G$ on $Y_L$ and $Y'_L$ equivalent?
    
    \item Is the action of $G$ on $Y_L$ acylindrical? If so, is it largest? (This is a stronger statement than the previous question.)
    
    \item Does the induced topology $\mathcal{T}(G, Y, Y_L)$ on $\partial_{\kappa} Y$ coincide with the visual topology?
    
    \item Which $\CAT$ groups act on $\CAT$ spaces such that the $L$-separation spaces stabilize?
\end{enumerate}

\end{queintro}

\subsubsection*{Acknowledgments}

The authors are grateful to Abdul Zalloum for several discussions and explaining Theorem \ref{thm:GenevoisABDequivalence} to them. The second author thanks Fran\c{c}ois Dahmani, Alessandro Sisto, and Davide Spriano for several discussions and suggestions. The first author was partially supported by NSF Grant DMS-2106906. The second author was partially funded by the SNSF grant 194996.

\section{Preliminaries} \label{sec:Preliminaries}




\subsection{Hyperbolic spaces} \label{subsec:hyperbolicspaces}
Given a metric space $X$, a {\it geodesic} is an isometric embedding $\gamma \colon I \rightarrow X$, where $I$ is a (possibly unbounded) interval in $\mathbb{R}$. The embedding $\gamma$ is a {\it geodesic ray} if $I$ is unbounded on one side, and a {\it bi-infinite geodesic} if $I=\mathbb R$.  A map $f \colon X \rightarrow X'$ between metric spaces is a {\it coarsely Lipschitz embedding} if there are constants $K \geq 1$ and $C \geq 0$ such that for all $x, y \in X$, we have
\[ d(f(x), f(y)) \leq Kd(x,y) + C. \]
The map $f$ a {\it $(K,C)$-quasi-isometric embedding} if for all $x, y \in X$, we have
\[ \frac{1}{K} d(x,y) - C \leq d(f(x),f(y)) \leq Kd(x,y) + C. \]
If $f$ is a $(K, C)$-quasi-isometric embedding and for every $x' \in X'$, there exists $x \in X$ such that $d(f(x), x') \leq C$, then $f$ is a quasi-isometry. A \textit{$(K,C)$-quasi-geodesic} is a $(K, C)$-quasi-isometric embedding $\gamma \colon I \rightarrow X$ of an interval in $\mathbb{R}$.

In this paper, we will frequently work with non-geodesic and non-proper hyperbolic spaces.  For this reason, we use the most general definition of a hyperbolic metric space, which involves Gromov products.  For more details, see \cite{BH, BS}.

Let $(X, d)$ be a metric space and $\delta \geq 0$. Given three points $x, y, o \in X$, the {\it Gromov product of $x$ and $y$ with respect to $o$} is
\[ (x \mid y)_o := \frac{1}{2} (d(o,x) + d(o,y) - d(x,y) ). \]
The space $X$ is \textit{$\delta$--hyperbolic} if for all $x, y, z, o \in X$, we have
\[ (x \mid y)_o \geq \min\{ (x \mid z)_o, (z \mid y)_o \} - \delta. \]
The visual boundary of a $\delta$-hyperbolic space $X$ is defined as follows. A sequence  $(x_i)_i$ in $X$ \textit{converges to infinity} if
\[ \lim_{i, j \rightarrow \infty} (x_i \mid x_j)_o = \infty. \]
This property is independent of the choice of $o$. Two sequences $(x_i)_i$ and $(y_i)_i$ converging to infinity are {\it equivalent} if
\[ \lim_{i \rightarrow \infty} (x_i \mid y_i)_o = \infty. \]
The {\it boundary at infinity} or \textit{visual boundary} $\partial_{\infty}X$ of $X$, as a set, is the set of all equivalence classes of sequences converging to infinity.

In addition to sequences, we will often consider paths $\gamma \colon [a, \infty) \rightarrow X$ that represent a point in $\partial_{\infty} X$ in some sense. We make this precise here.

\begin{definition} \label{def:Inducingapointatinfinity}
Let $X$ be a $\delta$-hyperbolic metric space.  A map  $\gamma \colon [a, \infty) \rightarrow X$  {\it defines a point in $\partial_{\infty} X$} if there exists a point $\xi \in \partial_{\infty}X$ such that for every increasing sequence $a \leq t_1 \leq t_2 \leq \dots$ with $\lim_{i \rightarrow \infty} t_i = \infty$, the sequence $\gamma(t_i)$  converges to infinity and $(x_i)_i \in \xi$. We write $\xi = [\gamma]$.
\end{definition}

It is common to consider quasi-geodesic rays $\gamma$ when studying the boundary at infinity, as these are particularly well-behaved in geodesic hyperbolic spaces. However, there are non-quasi-geodesic paths that define a point in $\partial_{\infty}X$. For example, a path in a $3$-valent tree that follows a geodesic ray but keeps making larger and larger detours defines the same point in the boundary as the geodesic ray, but is not a quasi-geodesic. For our purposes, we will need to consider such paths.\\

The Gromov product can be extended from a hyperbolic space $X$ to its boundary at infinity. For any two equivalence classes $\xi, \eta \in \partial_{\infty}X$ and $o \in X$, we define
\[ (\xi \mid \eta)_o := \inf_{(x_i)_i \in \xi, (y_i)_i \in \eta} \liminf_{i \rightarrow \infty} ( x_i \mid y_i)_o. \]
A basic property of hyperbolic spaces is that for any choice of representatives $(x_i)_i \in \xi$, $(y_i)_i \in \eta$, we have $\liminf_{i \rightarrow \infty} (x_i \mid y_i)_o \leq \limsup_{i \rightarrow \infty} (x_i \mid y_i)_o \leq ( \xi \mid \eta)_o + 2\delta$.

The Gromov product allows us to equip the boundary at infinity with a topology called the visual topology. Fix a basepoint $o \in X$ and let $\xi \in \partial_{\infty} X$, and $R \geq 0$. We define
\[ U_{o,R}(\xi) := \{ \eta \in \partial_{\infty} X \mid (\xi \mid \eta)_o > R \}. \]
Associating to every $\xi$ the collection of sets $\{ U_{o, R}(\xi) \mid R \geq 0 \}$ yields a neighbourhood basis on $\partial_{\infty} X$. The topology induced by these neighbourhoods does not depend on the choice of $o$, and we call it the {\it visual topology}. A well-known fact about the visual topology is that it is metrizable by a family of metrics called visual metrics. From now on, we will always consider $\partial_{\infty} X$ equipped with this topology.




\subsection{Largest acylindrical actions} \label{subsec:Acylindricalactions}

The most important examples for hyperbolic projections will be given by largest acylindrical actions. In this section, we provide the basic definitions and some important results that we will use later on. Before we talk about acylindrical actions, we have some general remarks regarding group actions on metric spaces. Throughout this paper, we assume that all our actions are by isometries.

\begin{definition} \label{def:Coarselyequivariant}
Let $G$ be a group acting  on two metric spaces $X, Y$. We call a map $f \colon Y \rightarrow X$ {\it coarsely $G$-equivariant} if there exists a constant $C$ such that for all $g \in G$ and all $y \in Y$, we have $d( f(gy), gf(y)) \leq C$.
\end{definition}

\begin{remark} \label{rem:InducedProjection}
Suppose a group $G$ acts by isometries on two spaces $X, Y$ and suppose the action on $Y$ is geometric. We will frequently use the fact that there is a coarsely $G$-equivariant, coarsely Lipschitz map $p \colon Y \rightarrow X$ that is unique up to bounded distance. Indeed, since $G$ acts on $Y$ geometrically, there is a coarsely $G$-equivariant quasi-isometry from $Y$ to $(G,d_S)$, where $d_S$ is a word metric with respect to some finite generating set. Fixing a basepoint $x_0 \in X$, the orbit map $(G,d_S)\to X$ defined by $g \mapsto gx_0$ is a $G$-equivariant and Lipschitz. The composition of these two maps is the desired map $p$. It is straightforward to check that any two such maps have bounded distance from each other, using that the action of $G$ on $Y$ is  cocompact. We call $p$ a projection map induced by the action $G \curvearrowright X$. If the action is clear from the context, we refer to $p$ as the induced projection map.
\end{remark}

We now define acylindrical actions. The following definition goes back to \cite{Bowditch07} and \cite{Osin15}.

\begin{definition}
An isometric action of a group $G$ on a metric space $X$  is  {\it acylindrical} if for every $\epsilon > 0$, there exist $R, N \geq 0$ such that for all $x, y \in X$ with $d_X(x, y) \geq R$, we have
\[ \vert \{ g \in G \mid d_X( x, gx) \leq \epsilon,\, d_X( y, gy) \leq \epsilon \} \vert \leq N. \]
\end{definition}

The set of isometric actions of a given group admits a preorder, introduced in \cite{AbbottBalasubramanyaOsin19}.

\begin{definition} \label{def:Dominated}
Let $G$ be a group which acts coboundedly and by isometries on two metric spaces $X$ and $X'$. The action of $G$ on $X$ is {\it dominated} by the action of $G$ on $X'$ if there exists a coarsely $G$-equivariant, coarsely Lipschitz map $p \colon X' \rightarrow X$. We write $X \preceq X'$. We say two actions are {\it equivalent} if $X \preceq X'$ and $X' \preceq X$, and we write $X \sim X'$.
\end{definition}

It is straightforward to show that if  $G\curvearrowright X$ and $G\curvearrowright X'$ are  equivalent cobounded actions, then the coarsely $G$-equivariant, coarsely Lipschitz maps between them are quasi-isometries. Since $\preceq$ defines a preorder, it induces a partial order on the space of equivalence classes of cobounded actions of $G$.

\begin{definition}
Let $\AH(G)$ denote the set of $\sim$-equivalence classes of cobounded acylindrical actions of $G$ on $\delta$-hyperbolic spaces. An acylindrical action on a hyperbolic space is ${\it largest}$ if its equivalence class in $\AH(G)$ is largest with respect to the partial order defined by $\preceq$.
\end{definition}

Not every group admits a largest acylindrical action; the first author constructed the first such example in  \cite{Abbott16}. 
However, it follows from the definition that if a group has a largest acylindrical action, its equivalence class is unique. The following result guarantees existence of largest acylindrical actions for a large class of interesting examples; see Section~\ref{sec:Topologyforgroupswithacoarsemedianstructure} for further discussion of hierarchically hyperbolic groups.

\begin{theorem}[\cite{AbbottBehrstockDurham21}]\label{thm:HHGLargest}
Every hierarchically hyperbolic group admits a largest acylindrical action. This includes, in particular, groups that act properly, cocompactly, and cubically on a $\CAT$ cube complex that admits a factor system (see Definition \ref{def:factorsystem}).
\end{theorem}




\subsection{Sublinearly Morse boundaries} \label{subsec:sublinearlyMorseboundaries}

In this section, we review the definition and properties of $\kappa$-Morse geodesic rays. For further details, see \cite{QingRafiTiozzo}. Let $\kappa \colon [0, \infty) \rightarrow [0, \infty)$ be a monotone increasing, concave, and sublinear function, that is
\[
\lim_{t \to \infty} \frac{\kappa(t)} t = 0. 
\]
Let $Y$ be a metric space with base point $o \in Y$, and denote for all $y \in Y$
\[ \kappa(y) := \kappa( d(o,y) ). \]

\begin{definition} \label{Def:Neighborhood} 
Given a constant $n$ and a closed set $Z \subset Y$ the \textit{$(\kappa,n)$--neighborhood of $Z$} is 
\[
\mathcal{N}_\kappa(Z, n) = \Big\{ y \in Y \vert  d(y, Z) \leq  n \cdot \kappa(y)  \Big\}.
\]
We will frequently omit the precise value of $n$ and speak of $\kappa$-neighbourhoods.
\end{definition} 

\begin{figure}[h!]
\begin{tikzpicture}
 \tikzstyle{vertex} =[circle,draw,fill=black,thick, inner sep=0pt,minimum size=.5 mm] 
[thick, 
    scale=1,
    vertex/.style={circle,draw,fill=black,thick,
                   inner sep=0pt,minimum size= .5 mm},
                  
      trans/.style={thick,->, shorten >=6pt,shorten <=6pt,>=stealth},
   ]

 \node[vertex] (a) at (0,0) {};
 \node at (-0.2,0) {$o$};
 \node (b) at (10, 0) {};
 \node at (10.6, 0) {$\gamma$};
 \node (c) at (6.7, 2) {};
 \node[vertex] (d) at (6.68,2) {};
 \node at (6.7, 2.4){$y$};
 \node[vertex] (e) at (6.68,0) {};
 \node at (6.7, -0.5){$y_{\gamma}$};
 \draw [-,dashed](d)--(e);
 \draw [-,dashed](a)--(d);
 \draw [decorate,decoration={brace,amplitude=10pt},xshift=0pt,yshift=0pt]
  (6.7,2) -- (6.7,0)  node [black,midway,xshift=0pt,yshift=0pt] {};

 \node at (7.8, 1.2){$n \cdot \kappa(y)$};
 \node at (3.6, 0.7){$d(o,y)$};
 \draw [thick, ->](a)--(b);
 \path[thick, densely dotted](0,0.5) edge [bend left=12] (c);
\node at (1.4, 1.9){$\mc N_{\kappa}(\gamma,n)$};
\end{tikzpicture}
\caption{A $\kappa$-neighbourhood of a geodesic ray $\gamma$ with multiplicative constant $n$, where $y_{\gamma}$ is a point on $\gamma$ closest to $y$.}
\end{figure}
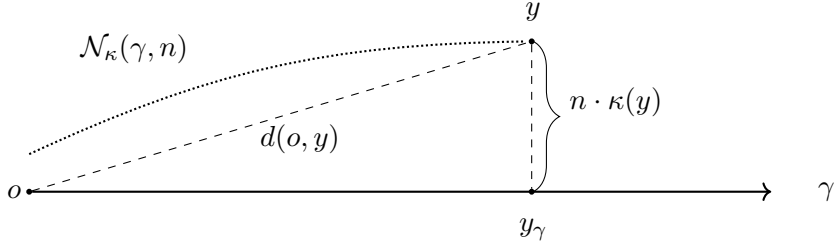

\begin{definition} \label{Def:Fellow-Travel}
Two quasi-geodesic rays $\alpha$ and $\beta$  in $Y$ \emph{$\kappa$--fellow travel} each other if $\alpha$ is contained in some 
$\kappa$--neighbourhood of $\beta$ and $\beta$ is contained in some 
$\kappa$--neighbourhood of $\alpha$.  This defines an equivalence
relation on the set of quasi-geodesic rays in $Y$.
\end{definition}

\begin{definition} \label{Def:Morse} 
A subset $Z\subseteq Y$ is \emph{$\kappa$--Morse} if there is a function
\[
\mu_Z \colon \mathbb{R}_+^2 \to \mathbb{R}_+
\]
so that for any $(K, C)$--quasi-geodesic $\beta \colon [s,t] \to Y$  with end points 
on $Z$,
\[
\beta([s,t])  \subset \mathcal{N}_{\kappa} \big(Z,  \mu_Z(K, C)\big). 
\]
We refer to $\mu_{Z}$ as the \emph{Morse gauge} for $Z$. For convenience, we may always assume that $\mu_Z(K,C)$ is the largest element in the set $\{K,C, \mu_Z(K,C) \}$. When $\kappa \equiv 1$, we recover the standard definition of a \emph{Morse set}. For a more detailed discussion of Morse sets, see \cite{Cordes}.
\end{definition}

We are most interested in the case where the set $Z$ is a quasi-geodesic ray. Such quasi-geodesic rays allow us to define the $\kappa$-Morse boundary.

\begin{definition} \label{def: SM}
Let $\kappa$ be a sublinear function. As a set, the \textit{$\kappa$-Morse boundary} of a geodesic metric space $Y$ is 
\[\partial_\kappa Y : = \{ \text{ all } \kappa\text{-Morse quasi-geodesic rays } \} / \kappa\text{-fellow travelling}\]
\end{definition}
Note that we do not require the $\kappa$-Morse rays in the above definition to begin at the same point, as two representatives with different starting points are contained in sufficiently large $\kappa$-neighborhoods of each other.

When the particular sublinear function $\kappa$ is not important, we simply call this boundary the \textit{sublinearly Morse boundary}.  Using this notation, the original Morse boundary defined by Cordes (see \cite{Cordes, CharneySultan, Murray}) is denoted $\partial_1 Y$.

The following two results are useful facts about $\kappa$-Morse boundaries.

\begin{lemma}[{\cite[Lemma 4.2]{QingRafiTiozzoII}}] \label{lem:GeodesicRepresentatives}
If $Y$ is a proper, geodesic metric space and $\xi \in \partial_{\kappa} Y$, then there exists a geodesic ray $\gamma \in \xi$.
\end{lemma}

\begin{lemma}[{\cite[Theorem A]{QingRafiTiozzoII}}]
Every quasi-isometry $\Psi : Y \rightarrow Y'$ between proper geodesic metric spaces extends to a bijection $\partial \Psi : \partial_{\kappa} Y \rightarrow \partial_{\kappa} Y'$ for every sublinear $\kappa$.
\end{lemma}

In \cite{QingRafiTiozzoII}, Qing-Rafi-Tiozzo define a topology on $\partial_{\kappa} Y$ that turns the bijection in the previous lemma into a homeomorphism. In other words, the $\kappa$-Morse boundary is a topological quasi-isometry invariant when equipped with that topology. This topology will not be necessary for this paper; we refer the interseted reader to \cite{QingRafiTiozzoII} for further details.




\subsection{$\CAT$ spaces} \label{subsec:CAT(0)spaces}

We  recall the following definitions and facts for $\CAT$ spaces, and refer the reader to \cite{BH} for the necessary definitions and properties.

Given a proper $\CAT$ space $Y$, two geodesic rays $\gamma : [a, \infty) \rightarrow Y$, $\gamma': [a', \infty) \rightarrow Y$ are {\it equivalent} if  there exists a constant $D$ such that
\[ \sup_{t \geq 0} d(\gamma(a + t), \gamma'(a' + t) ) < \infty. \]
The \textit{visual boundary} $\partial_{\infty} Y$ of $Y$ is the set of equivalence classes of geodesic rays. This boundary can be equipped with the visual topology,  defined as follows. Fix a base point $o$. Since $Y$ is proper, every point $\xi \in \partial_{\infty} Y$ has a unique representative $\gamma_{\xi} \in \xi$ that starts at $o$. For every $\epsilon, R > 0$ and $\xi \in \partial_{\infty} Y$, let 
\[ U_{o, \epsilon, R}(\xi) := \{ \eta \in \partial_{\infty}Y \vert d(\gamma_{\eta}(R), \gamma_{\xi}(R) ) < \epsilon \}. \]
For every $\xi$, the collection $\{ U_{o, \epsilon, R}(\xi) \}$ forms a neighbourhood basis. We define the visual topology on $\partial_{\infty} Y$ to be the topology defined by these neighbourhood bases. This topology is independent of the choice of $o$. From now on, we will always consider boundaries of $\CAT$ spaces equipped with this topology.\\

In contrast to the hyperbolic case, the visual boundary of a $\CAT$ space is not invariant under quasi-isometry. Indeed, Croke--Kleiner showed that there are groups acting geometrically on distinct $\CAT$ spaces (implying that these $\CAT$ spaces are coarsely $G$-equivariantly quasi-isometric), whose visual boundaries are not homeomorphic (see \cite{CrokeKleiner00}). This observation serves as the motivation for several boundary constructions, including the definition of the sublinearly Morse boundary.

If $Y$ is a proper, geodesic $\CAT$ space, we obtain a canonical inclusion $\partial_{\kappa} Y \hookrightarrow \partial_{\infty} Y$ as follows. Every equivalence class of $\kappa$-fellow traveling quasi-geodesic rays $[\beta]$ in $Y$ contains a geodesic representative \cite[Lemma 4.2]{QingRafiTiozzoII}, and any two such geodesic representatives are equivalent in the sense of the visual boundary $\partial_{\infty}Y$. The map thus send $[\beta]$ to the point in $\partial_{\infty} Y$ represented by all geodesic rays in $[\beta]$. One immediately checks from the definition that  geodesic rays in $Y$ that do not $\kappa$-fellow travel are not equivalent in $\partial_{\infty} Y$. This map is thus injective, yielding the desired inclusion. We may therefore consider the restriction of the visual topology on $\partial_{\infty} Y$ to $\partial_{\kappa} Y$, which we denote  by $\Vis$.

Cashen showed that $\Vis$ is not a quasi-isometry invariant by constructing two $\CAT$ spaces that are quasi-isometric, but whose $\kappa$-Morse boundaries are not homeomorphic \cite{Cashen}. His examples can easily be cubulated, so that quasi-isometry invariance does not even hold for $\CAT$ cube complexes. As we will see, this phenomenon disappears if one requires the metric spaces in question to admit certain properties; in particular they need to have sufficiently large isometry groups.




\section{The topology of hyperbolic projections} \label{sec:Thetopologyofhyperbolicprojections}
In this section we will define a topology on the $\kappa$--Morse boundary of a finitely generated group using a \textit{hyperbolic projection of $G$}. 

\begin{definition} \label{def:hyperbolicprojection}
Let $G$ be a finitely generated group, $Y$ a geodesic metric space and $X$ a hyperbolic metric space that admits a $D$-quasi-ruling. A {\it hyperbolic projection of $G$} is a triple $(G, Y, X)$ together with a geometric action $G \curvearrowright Y$ and a cobounded action $G \curvearrowright X$ by isometries.
\end{definition}

The condition that $X$ admits a quasi-ruling is a technical assumption on the space $X$.  Intuitively (though not precisely), a space with a $D$--quasi-ruling has the property that there is a path between any two points that behaves well under reparametrization. (Here, a reparametrisation is an orientation-preserving homeomorphism $\varphi \colon I \rightarrow I$ on the domain of the path.) This notion is closely related to the existence of an unparametrized quasigeodesic between any two points: a path $\gamma \colon I \rightarrow X$ is an {\it unparametrised quasi-geodesic}, if it can be reparametrised as a quasi-geodesic $\gamma \circ \varphi \colon I \rightarrow X$. See Appendix \ref{sec:unparametrisedquasigeodesicsandhyperbolicspaces} for a in-depth discussion of quasi-rulings, their properties, and their relationship with unparametrised quasi-geodesics.

Recall that there is a coarsely $G$-equivariant, coarsely Lipschitz map $p \colon Y \rightarrow X$ which is unique up to bounded distance, which we call the induced projection map (see Remark \ref{rem:InducedProjection}). If $G$ acts coboundedly on $X$, then $p$ is coarsely surjective; that is, there exists a constant $D$ such that $d_{Haus}(X, p(Y)) \leq D$.
 
The case $\kappa\equiv 1$ is more straightforward than the case for a general sublinear $\kappa$.  We first deal with $\kappa \equiv 1$ in Section \ref{subsec:Morsecase} and put a topology on the Morse boundary of $G$, and then we deal with the general case in Section \ref{subsec:sublinearcase}.

\subsection{The topology for Morse boundaries} \label{subsec:Morsecase}
In this section, we require the hyperbolic projection $(G,X,Y)$ to satisfy an additional condition.

\begin{definition} \label{def:morseinjectivity}
Let $(G, Y, X)$ be a hyperbolic projection of $G$ and $p \colon Y \rightarrow X$ an induced projection map. We say $(G, Y, X)$ is {\it Morse-injective} if for all unbounded Morse geodesics $\gamma$ in $Y$, the image $p \circ \gamma$ is an unbounded, unparametrised quasi-geodesic in $X$.
\end{definition}

It is straightforward to check that this definition is independent of the choice of $p$, since $p$ is unique up to bounded distance. 

\begin{lemma}\label{lem:inducedinclusion}
If $(G,Y,X)$ is a Morse-injective hyperbolic projection, there is a $G$-equivariant injective map $\partial p\colon \partial_{1} Y \hookrightarrow \partial_{\infty} X$.
\end{lemma}

\begin{proof}
By Lemma \ref{lem:GeodesicRepresentatives}, every point $\xi \in \partial_{1} Y$  can be represented  by a Morse geodesic ray in $Y$, which we  denote  $\gamma$. By Morse-injectivity, $p \circ \gamma$ can be reparametrised to be an unbounded quasi-geodesic in the hyperbolic space $X$, and therefore $p\circ \gamma$ defines a point in $\partial_{\infty} X$. If $\sigma$ is a Morse quasi-geoedesic ray in $Y$ also representing $\xi$, then the Hausdorff distance between $\sigma$ and $\gamma$ is finite.  The map $p$ is coarsely Lipschitz, and so $\sigma$ and $\gamma$  project to subsets at finite Hausdorff distance in $X$. In particular, the quasi-geodesic $\sigma$ defines  the same point in $\partial_{\infty}X$ as $\gamma$. Thus there is a well-defined map $[\gamma] \mapsto [p \circ \gamma]$ from $\partial_{1} Y$ to $\partial_{\infty} X$. We denote this map by $\partial p$.

Any two distinct points in $\partial_{1} Y$ can be connected by a bi-infinite Morse geodesic, which projects to an  unparametrised quasi-geodesic in $X$ that is unbounded in both directions. After reparametrisation, this projection is a bi-infinite quasi-geodesic, and hence represents two distinct points in $\partial_{\infty} X$. Therefore, $\partial p$ has to be injective.  Finally, the map $\partial p$ is $G$-equivariant as $p$ is coarsely $G$-equivariant.
\end{proof}

We are now ready to define a topology on $\partial_1 G$.
\begin{definition}
Let $(G, Y, X)$ be a Morse-injective hyperbolic projection and $\partial p \colon \partial_{1} Y \hookrightarrow \partial_{\infty} X$ the induced inclusion as in Lemma \ref{lem:inducedinclusion}. The {\it topology of the hyperbolic projection}, denoted $\mathcal{T}(G, Y, X)$, is  the pull-back of the visual topology on $\partial_{\infty} X$ under $\partial p$ to $\partial_{1} Y$. 
\end{definition}

We emphasize that this topology depends not only on $G, Y,$ and  $X$, but also on the specific actions of $G$ on $Y$ and $X$.\\

We now show how this topology changes when we change the metric spaces $X$ and $Y$. Recall that  two cobounded actions of a group $G$ on spaces $X, X'$ are equivalent, denoted $X \sim X'$, if and only if there exists a coarsely $G$-equivariant quasi-isometry $\Phi \colon X \rightarrow X'$.

\begin{proposition} \label{prop:CommutingDiagram}
Let $(G, Y, X)$ and $(G, Y', X')$ be two hyperbolic projections such that $X \sim X'$, and let $p \colon Y \rightarrow X$ and $p' \colon Y' \rightarrow X'$ be the induced projection maps. There are coarsely $G$-equivariant quasi-isometries $\Psi \colon Y \rightarrow Y'$ and $\Phi \colon X \rightarrow X'$ such that the following statements hold:

\begin{enumerate}
    \item $(G, Y, X)$ is Morse-injective if and only if $(G, Y', X')$ is Morse-injective.
    
    \item If $(G, Y, X)$ is Morse-injective, then the following diagram commutes:
    \[
  \begin{tikzcd}
    \partial_{1} Y \arrow{r}{\partial \Psi} \arrow[hook, swap]{d}{\partial p} & \partial_{1} Y' \arrow[hook]{d}{\partial p'} \\
     \partial_{\infty} X \arrow{r}{\partial \Phi} & \partial_{\infty} X'
  \end{tikzcd}
\]

    \item The map
    \[ \partial \Psi \colon ( \partial_{1} Y, \mathcal{T}(G, Y, X)) \rightarrow (\partial_{1} Y', \mathcal{T}(G, Y', X') ) \]
    is a homeomorphism
\end{enumerate}
\end{proposition}

Before proving this proposition, we show how it implies that $\mc T(G,Y,X)$ is, in fact, a topology on the Morse boundary of $G$.
\begin{corollary} \label{cor:TopologyonGroup}
The topology $\mathcal{T}(G, Y, X)$ does not depend on the geometric action $G \curvearrowright Y$. It thus defines a topology on the Morse boundary of the group $G$. Furthermore, it only depends on the equivalence class of $X$. We write $\mathcal{T}(G, X)$ or $\mathcal{T}(G, [X])$ for the topology on $\partial_{1} G$ induced by a Morse-injective hyperbolic projection.
\end{corollary}

\begin{proof}[Proof of Corollary \ref{cor:TopologyonGroup}]
Suppose $G$ acts geometrically on two geodesic metric spaces $Y$  and $Y'$, and suppose $(G, Y, X)$ and $(G, Y', X)$ are Morse-injective hyperbolic projections. There exists a unique $G$-equivariant bijection $\Psi \colon \partial_{1} Y \rightarrow \partial_{1} Y'$. By Proposition \ref{prop:CommutingDiagram}, this map is a homeomorphism with respect to $\mathcal{T}(G, Y, X)$ and $\mathcal{T}(G, Y', X)$. The Morse boundary of the group $G$ is, by definition, the collection of sets $\partial_{1} Y$ for all geometric actions $G \curvearrowright Y$ identified by these unique $G$-equivariant bijections. Since they are all homeomorphisms, $\mathcal{T}(G, Y, X)$ defines a topology on $\partial_{1} G$ independent of $Y$.
\end{proof}

We now prove the proposition.
\begin{proof}[Proof of Proposition \ref{prop:CommutingDiagram}]

Let $G$ be a group that acts geometrically on two geodesic metric spaces $Y$, and $Y'$ and coboundedly isometrically on two geodesic hyperbolic metric spaces $X$ and $ X'$. Denote the induced projection maps by $Y \xrightarrow{p} X$, $Y' \xrightarrow{p'} X'$. There exist coarsely $G$-equivariant quasi-isometries $\Psi \colon Y \rightarrow Y'$ and $\Phi \colon X \rightarrow X'$ due to the Milnor-Schwarz Lemma and the assumption that $X \sim X'$, respectively. We obtain the following diagram of coarsely $G$-equivariant maps:

\[
  \begin{tikzcd}
    Y \arrow{r}{\Psi} \arrow[swap]{d}{p} & Y' \arrow{d}{p'} \\
     X \arrow{r}{\Phi} & X'
  \end{tikzcd}
\]

Since all maps in this diagram are coarsely $G$-equivariant and $G$ acts cocompactly on $Y$, the diagram commutes up to bounded error. Indeed, choose a bounded set $K \subset Y$ such that $GK = Y$. Since $K$ is bounded and all maps in the diagram are coarsely Lipschitz, we see that $\Phi \circ p(K)$ and $p' \circ \Psi(K)$ are both bounded sets and thus $d( \Phi \circ p(y), p' \circ \Psi(y))$ is uniformly bounded for all $y \in K$. Since both maps are coarsely $G$-equivariant and $GK = Y$, we obtain a uniform bound for all $y \in Y$, making the diagram commute up to bounded error.

We first show statement (1). Suppose $(G, Y, X)$ is Morse-injective, that is, all unbounded Morse geodesics in $Y$  project to unbounded, unparametrised quasi-geodesics in $X$. Let $\gamma'$ be a Morse quasi-geodesic ray in $Y'$  respresenting a point $\xi \in \partial_{\kappa} Y'$. Choosing a quasi-inverse $\overline{\Psi}$ of $\Psi$, we obtain a Morse quasi-geodesic ray $\overline{\Psi} \circ \gamma'$ in $Y$ representing $(\partial \Psi)^{-1}(\xi)$.  This ray must  be finite Hausdorff distance from a Morse geodesic ray $\gamma$ representing $(\partial \Psi)^{-1}(\xi)$. Since $(G,Y,X)$ is Morse-injective, $p \circ \bar{\Psi} \circ \gamma'$ has bounded Hausdorff distance from the unbounded unparametrised quasi-geodesic $p \circ \gamma$ in $X$.  Since $p \circ \gamma$ represents the point $\partial p \circ (\partial \Psi)^{-1}(\xi)$,  the ray $p \circ \bar{\Psi} \circ \gamma'$ is also a representative of $\partial p \circ (\partial \Psi)^{-1}(\xi)$. Applying the quasi-isometry $\Phi$, we obtain an unbounded, unparametrised quasi-geodesic $\Phi \circ p \circ \bar{\Psi} \circ \gamma'$ in $X'$, which has uniformly bounded distance from $p' \circ \gamma'$. We conclude that $p' \circ \gamma'$ has to be an unbounded, unparametrised quasi-geodesic, and therefore defines a point in $\partial_{\infty} X'$. An analogous argument applies if $\gamma'$ is a bi-infinite Morse geodesic in $Y'$. We conclude that $(G, Y', X')$ is Morse-injective.

Next, we prove (2). Since $\Phi \circ p$ and $p' \circ \Psi$ differ by a uniformly bounded distance, their induced maps on boundaries commute point-wise, yielding the following commuting diagram:

\[
  \begin{tikzcd}
    \partial_{1} Y \arrow{r}{\partial \Psi} \arrow[hook, swap]{d}{\partial p} & \partial_{1} Y' \arrow{d}{\partial p'} \\
     \partial_{\infty} X \arrow{r}{\partial \Phi} & \partial_{\infty} X'
  \end{tikzcd}
\]

It remains to prove (3). By definition $\mathcal{T}(G, Y, X)$ and $\mathcal{T}(G, Y', X')$ are the topologies on $\partial_{1} Y$ and $\partial_{1} Y'$ obtained by viewing them as subspaces of $\partial_{\infty}X$ and $\partial_{\infty}X'$, respectively, equipped with the visual topology. Since the diagram in (2) commutes, we see that $\partial \Psi$ can be seen as the restriction of $\partial \Phi$ to subspaces. The map $\partial \Phi$ is a homeomorphism by Corollary \ref{cor:quasiisometriesofquasiruledspacesextend}, and thus its restriction to $\partial \Psi : \partial_{1} Y \rightarrow \partial_{1} Y'$ is a homeomorphism in the induced subspace topology. This proves the proposition.
\end{proof}




\subsection{The topology for sublinearly Morse boundaries} \label{subsec:sublinearcase}

When considering the $\kappa$-Morse boundary for some sublinear function different from the constant function, the function $\partial p$ from Lemma \ref{lem:inducedinclusion} is no longer well-defined.
The key issue is that, given a point $\xi \in \partial_{\kappa} Y$ and two $\kappa$-Morse quasi-geodesics in $Y$ representing $\xi$, these two quasi-geodesics are only $\kappa$-close to each other, but they may not have finite Hausdorff distance. It is not clear whether two $\kappa$-close quasi-geodesics project to unparametrised quasi-geodesics with finite Hausdorff distance, and thus we do not know if their projections define the same point in the visual boundary.  
We resolve this by replacing the definition of Morse-injectivity with a stronger requirement.

\begin{definition}
Let $\kappa$ be a sublinear function. A hyperbolic projection $(G, Y, X)$  is {\it $\kappa$-injective} if the following two conditions hold.  The projection $p\circ \gamma$ of every $\kappa$-Morse quasi-geodesic ray $\gamma$ in $Y$ defines a point in $\partial_\infty X$; and two $\kappa$-Morse quasi-geodesic rays in $Y$ are $\kappa$-close if and only if their projections define the same point in $\partial_\infty X$.
\end{definition}

Recall that a path in a hyperbolic space defines a point $\xi$ at infinity if every increasing sequence of points on the path is a representative of $\xi$ (see Definition \ref{def:Inducingapointatinfinity}).  If a hyperbolic projection $(G, Y, X)$ is $\kappa$-injective, define a map $\partial p \colon \partial_{\kappa} Y \hookrightarrow \partial_{\infty} X$  that sends any $\kappa$-Morse quasi-geodesic ray $\gamma$ in $Y$ to the point in $\partial_{\infty} X$ represented by its projection $p \circ \gamma$.
Using this stronger condition, an analogous argument as for Lemma ~\ref{lem:inducedinclusion} shows that this map is well-defined and injective.

\begin{lemma}\label{lem:kappainducedinclusion}
If a hyperbolic projection $(G, Y, X)$ is $\kappa$-injective, there is an injective map $\partial p \colon \partial_{\kappa} Y \hookrightarrow \partial_{\infty} X$.
\end{lemma}

Using the inclusion $\partial p$, we can pull back the visual topology on $\partial_{\infty} X$ to $\partial_{\kappa} Y$. As in the case $\kappa \equiv 1$, we denote this topology by $\mathcal{T}(G, Y, X)$. We can now prove an analogous result to Proposition \ref{prop:CommutingDiagram}.

\begin{proposition} \label{prop:CommutingDiagramsublinearcase}
Let $(G, Y, X)$ and $(G, Y', X')$ be two hyperbolic projections such that $X \sim X'$, and let $p \colon Y \rightarrow X$ and $p' \colon Y' \rightarrow X'$ be the induced projection maps. Let $\Psi \colon Y \rightarrow  Y'$ and $\Phi \colon X \rightarrow X'$ be coarsely $G$-equivariant quasi-isometries induced by the geometric actions on $Y, Y'$ and the fact that $X \sim X'$. The following statements hold.
\begin{enumerate}
    \item $(G, Y, X)$ is $\kappa$-injective if and only if $(G, Y', X')$ is $\kappa$-injective.
    
    \item If $(G, Y, X)$ is $\kappa$-injective, then the following diagram commutes:
    \[
  \begin{tikzcd}
    \partial_{\kappa} Y \arrow{r}{\partial \Psi} \arrow[hook, swap]{d}{\partial p} & \partial_{\kappa} Y' \arrow[hook]{d}{\partial p'} \\
     \partial_{\infty} X \arrow{r}{\partial \Phi} & \partial_{\infty} X'
  \end{tikzcd}
\]

    \item The map
    \[ \partial \Psi \colon ( \partial_{\kappa} Y, \mathcal{T}(G, Y, X)) \rightarrow (\partial_{\kappa} Y', \mathcal{T}(G, Y', X') ) \]
    is a homeomorphism.
    
\end{enumerate}
\end{proposition}

\begin{proof}
As in the proof of Proposition \ref{prop:CommutingDiagram}, we obtain that the following diagram commutes up to bounded distance.
\[
  \begin{tikzcd}
    Y \arrow{r}{\Psi} \arrow[swap]{d}{p} & Y' \arrow{d}{p'} \\
     X \arrow{r}{\Phi} & X'
  \end{tikzcd}
\]

Suppose $(G, Y, X)$ is $\kappa$-injective. Let $\gamma'$ be a $\kappa$-Morse quasi-geodesic ray in $Y'$, and let $\overline{\Psi}$ be a quasi-inverse of $\Psi$. Then $\overline{\Psi} \circ \gamma'$ is a $\kappa$-Morse quasi-geodesic ray in $Y$, and so $p \circ \overline{\Psi} \circ \gamma'$ defines a point in $\partial_{\infty} X$ by the definition of $\kappa$-injectivity.  It then follows from Corollary \ref{cor:quasiisometriesofquasiruledspacesextend} that $\Phi \circ p \circ \overline{\Psi} \circ \gamma'$ defines a point in $\partial_{\infty} X'$. Since the above diagram commutes up to bounded distance, $p' \circ \gamma'$ has finite Hausdorff distance from $\Phi \circ p \circ \overline{\Psi} \circ \gamma'$, and hence they represent the same point in $\partial_{\infty} X'$. Therefore $(G, Y', X')$ is $\kappa$-injective. Statements (2) and (3) then follow by the same argument as in Proposition \ref{prop:CommutingDiagram}.
\end{proof}

\begin{remark} \label{rem:InducedMetric}
The embedding of the $\kappa$-Morse boundary of a group $G$ into the visual boundary a $D$-quasi-ruled hyperbolic space $X$ has a geometric implication, as well. Namely, one can pull-back the visual metrics on $\partial_\infty X$ to the $\kappa$-Morse boundary $\partial_{\kappa} G$. Therefore the topology induced by a $\kappa$-injective hyperbolic projection is metrizable, and we have an explicit description of a family of metrics in terms of the $\kappa$-injective hyperbolic projection. In an analogous way, the $\kappa$-Morse boundary inherits a cross ratio from a $\kappa$-injective hyperbolic projection.
\end{remark}

\begin{definition} \label{def:kappainjectivity}
 Proposition~\ref{prop:CommutingDiagramsublinearcase} implies that when $G\curvearrowright X$ is a cobounded action on a hyperbolic space that admits a quasi-ruling, the $\kappa$-injectivity of $(G, Y, X)$ depends only on $G$ and $[X]$. Therefore, we say that \textit{$(G, [X])$ is a $\kappa$-injective hyperbolic projection} whenever $(G, Y, X)$ is for any choice of $Y$ and $X$.  Moreover, since we have shown the topology $\mc T(G,Y,X)$ does not depend on the choice of $Y$, when $G$ is clear from context, we write simply $\mc T([X])$.
\end{definition}

It is natural to ask whether there are particular equivalence classes of actions that are in some sense `canonical' or `best.' An important candidate for this turns out to be the class of the largest acylindrical action of a group $G$. This requires that the group $G$ admits a largest acylindrical action on a hyperbolic space $X$, and that $(G,[X])$ is $\kappa$-injective. When that happens, we frequently call the topology of Definition \ref{def:kappainjectivity} the {\it acylindrical topology}. In the next subsection, we will focus on one particular class of groups for which the acylindrical topology is defined: cubulable groups.




\subsection{The case of cubulated groups} \label{subsec:Thecaseofcubulatedgroups}

Suppose $G$ acts geometrically and cubically on a $\CAT$ cube complex $Y$ that admits a factor system (see Definition \ref{def:factorsystem}). Such a group is also hierarchically hyperbolic. However, the cubulation provides more structure than a general hierarchically hyperbolic group has.  Thus in the more general setting of hierarchical hyperbolicity, which we discuss in Section~\ref{sec:Topologyforgroupswithacoarsemedianstructure}, we need to add additional conditions on the projection maps in order to understand the topology.   In this section, we show that the largest acylindrical action of $G$ is $\kappa$-injective.  Moreover,  we can describe the induced topology on $\partial_\kappa G$ very concretely. 

We begin with some basic terminology and facts about  $\CAT$ cube complexes. For a more detailed introduction to $\CAT$ cube complexes, see \cite{Sageev14}.

\begin{definition}
For $n \geq 0$, let $[0,1]^n$ be an $n$-cube equipped with the Euclidean metric. A \emph{face} of a $n$--cube is obtained by choosing some indices in $\{ 1, \dots, n \}$ and considering the subset of all points where, for each chosen index $i$, we fix the $i$-th coordinate  to be either zero or one. A \emph{cube complex} is a topological space obtained by gluing cubes together along faces, i.e., every gluing map is an isometry between faces.
\end{definition}

\begin{definition}
Let $Y$ be a cube complex and $y \in Y^{(0)}$ a vertex. The \emph{degree} of $y$, denoted $\deg(y)$, is  the number of edges incident to $y$. The cube complex $Y$ is {\it finite dimensional} if there exists $N \in \mathbb{N}$ such that every cube in $Y$ has dimension at most $N$.
\end{definition}

Any cube complex can be equipped with a metric in the following manner. Each $n$--cube is equipped with the Euclidean metric.  This allows one to define the length of continuous paths inside the cube complex: Simply partition every path into segments that lie entirely within one cube and use the Euclidean metric of that cube. Then
\[ d(x,y) := \inf \{ length(\gamma) \mid \gamma \text{ a continous path from $x$ to $y$} \} \]
defines a metric on $Y$, which is sometimes called  the $\ell^2$-metric.

\begin{definition}
A cube complex $Y$ is a $\CAT$ cube complex if $Y$ equipped with the $\ell^2$-metric is a $\CAT$ space
\end{definition}

\begin{definition}
A \emph{midcube} of cube $[0,1]^n$ is a subspace obtained by restricting exactly one coordinate in $[0,1]^n$ to $\frac{1}{2}$.
\end{definition}

\begin{definition}
Let $Y$ be a $\CAT$ cube complex. A \emph{hyperplane} is a connected subspace $h \subset Y$ such that for each cube $c$ of $Y$, the intersection $h \cap c$ is either empty or a midcube of $c$.  For each hyperplane $h$, the complement $Y \setminus h$ has exactly two components, called \emph{halfspaces} associated to $h$. We usually denote these $h^+$ and $h^-$. A hyperplane $h$ is said to \emph{separate} the sets $U,V \subseteq Y$ if $U \subseteq h^+$ and $V \subseteq h^-.$
\end{definition}

Whenever $Y$ is a $\CAT$ cube complex, we refer to the $\ell^2$-metric  as the \emph{CAT(0)-metric}. In contrast to the $\CAT$ metric, we can also equip every $n$--cube with the restriction of the $\ell^1$ metric of $\mathbb{R}^n$ and consider the induced path metric $d^{(1)}(\cdot, \cdot)$. We refer to $d^{(1)}$ as the {\it combinatorial metric} (or {\it $\ell^1$--metric}). The following lemma is standard; see, for example, \cite{CapraceSageev11}.

\begin{lemma} \label{lem: quasi-isometry between combinatorial and CAT(0) metrics}
If $Y$ is a finite-dimensional $\CAT$ cube complex, then the $\CAT$ metric $d$ and the combinatorial metric $d^{(1)}$ are bi-Lipschitz equivalent and complete. In particular, if all cubes in $Y$ have dimension $\leq m$, then $d \leq d^{(1)} \leq \sqrt{m} d$. Furthermore, for two vertices $x,y \in Y^{(0)}$, we have
\[ d^{(1)}(x,y)=|\{\text{hyperplanes $h \subseteq Y$ which separate the vertices }x,y\}|. \]
\end{lemma}

Throughout this paper, when we refer to geodesics in a $\CAT$ cube complex, these are geodesics with respect to the $\CAT$ metric.


    


\begin{definition}\label{def: chain} A \emph{chain} in $Y$ is a (possibly infinite) collection of mutually disjoint hyperplanes $(h_i)_i$ which are associated to a collection of nested half-spaces $h_1^{+} \supset h_2^{+} \supset \dots$. 

\end{definition}

\begin{definition}[\cite{BehrstockHagenSisto2017a}]
\label{def:factorsystem}Let $Y$ be a finite-dimensional $\CAT$ cube complex. A \emph{factor system}, denoted $\mathfrak{F}$, is a collection of subcomplexes 
of $Y$ such that:

\begin{enumerate}
    \item $Y \in \mathfrak{F}$.
    \item Each $F \in \mathfrak{F}$ is a nonempty convex subcomplex of $Y$
    \item There exists $C_{1} \geq 1$ such that for all $y \in Y^{(0)}$, at most $C_{1}$ elements of $\mathfrak{F}$ contain $y$.
    \item Every nontrivial convex subcomplex parallel to a hyperplane of $Y$ is in $\mathfrak{F}$. (A subcomplex is parallel to a hyperplane if every other hyperplane intersects  both or neither.)
    \item There exists $C_{2}$ such that for all $F, F' \in \mathfrak{F}$, either $\pi_{F}(F') \in \mathfrak{F}$ or $\diam(\pi_{F}(F')) \leq C_{2}$, where $\pi_F$ denotes the combinatorial closest-point projection in a cube complex.
\end{enumerate}

\end{definition}

Since every $\CAT$ cube complex $Y$ is also a geodesic $\CAT$ space, its visual boundary is well-defined. In this case, there is a topology on $\partial_\infty Y$ that is defined in terms of hyperplanes.

\begin{definition}\label{def:HYP}
Let $Y$ be a $\CAT$ cube complex, fix a vertex $o \in Y$, and let $h_1, \dots, h_n$ be distinct hyperplanes in $Y$. Define
\begin{equation*}
\begin{split}
V_{o, h_1, \dots, h_n} := \{ \xi \in \partial_{\infty} Y | \text{ the uni} & \text{que geodesic representative of $\xi$ based}\\
& \text{ at $o$ crosses the hyperplanes $h_1, \dots, h_n$} \}.
\end{split}
\end{equation*}

The collection $B=\{ V_{o, h_1, \dots, h_n}|  n \in \mathbb{N}, \,h_1,h_2,..,h_n \text{ are hyperplanes} \}$ forms the basis of a topology on $\partial_{\infty} Y$ which we denote $\Hyp$.
\end{definition}

This hyperplane topology  provides a particularly useful description of the visual topology on $\partial_\kappa Y$.

\begin{theorem}[{\cite[Theorem 1.3]{MediciZalloum21}}] \label{thm:HYPVIS}
Let $Y$ be a finite dimensional $\CAT$ cube complex and $\kappa$ a sublinear function. Then the restrictions of $\Hyp$ and $\Vis$ to $\partial_{\kappa} Y$ coincide.
\end{theorem}

We now introduce some notions that are helpful to detect $\kappa$-Morse geodesic rays in $\CAT$ cube complexes.

\begin{definition}
 Three hyperplanes  form a {\it facing triple}  if they are mutually disjoint and no one  separates the other two.
\end{definition}

Notice that if a geodesic $\gamma$ crosses three disjoint hyperplanes $h, k$, and $l$, in that order, then $k$ separates $h$ and $l$. In particular, a geodesic $\gamma$ cannot cross a facing triple.

\begin{definition} \label{def:wellseparatedhyperplanes}
Let $L \in \mathbb{N}$. Two hyperplanes $h, h'$ are {\it $L$-well-separated} if they are disjoint and any collection of hyperplanes intersecting both $h$ and $h'$ that contains no facing triple has cardinality at most $L$. We say $h, h'$ are {\it well-separated} if they are $L$-well-separated for some $L$.
\end{definition}

These definitions are motivated by the following result.

\begin{theorem}[{\cite[Theorem 2.35]{MediciZalloum21}}] \label{thm:Morse iff contracting iff excursion}
Let $Y$ be a proper $\CAT$ space and let $\zeta$ be an equivalence class of $\kappa$-fellow travelling quasi-geodesics in $Y$. If $\zeta$ contains a $\kappa$-Morse quasi-geodesic ray, then $\zeta$ contains a unique geodesic ray $\gamma$ starting at $\gamma(0)$. Furthermore, if $\zeta$ is an equivalence class of $\kappa$-fellow travelling quasi-geodesics, then the following are equivalent:

\begin{itemize}
    \item $\zeta$ contains a $\kappa$-Morse geodesic ray.

    \item Every quasi-geodesic ray in $\zeta$ is $\kappa$-Morse.
     \item There exists a quasi-geodesic ray in $\zeta$ which is $\kappa$-Morse.

\end{itemize}
Furthermore, if a geodesic ray $\gamma$ is $\kappa$-Morse, then there exists a constant $c$ such that $\gamma$ crosses a chain of hyperplanes $(h_i)_i$ in points $\gamma(t_i)$ such that $h_i, h_{i+1}$ are $c\kappa(t_{i+1})$-well-separated.

\end{theorem}

We now introduce Genevois's construction of a family of metrics on a $\CAT$ cube complex $Y$.

\begin{definition} [\cite{Genevois20b}]\label{def:GenevoisConstruction}
For any $L \in \mathbb{N}$,  define a metric $d_L$ on the set of vertices $Y^{(0)}$ by letting $d_L(x, y)$ be the size of the largest family of $L$-well-separated hyperplanes that separate $x$ from $y$.
\end{definition}

Genevois shows that, with each such metric, metric, $Y^{(0)}$ is a hyperbolic space.

\begin{proposition}[\cite{Genevois20b}]
Let $Y$ be a finite-dimensional $\CAT$ cube complex and $L \in \mathbb{N}$. Then $(Y^{(0)}, d_L)$ is $9(L+2)$-hyperbolic.
\end{proposition}

Let $Y_L$ be the metric space $(Y^{(0)}, d_L)$. Since the set of $Y_L$ is the set of vertices of $Y$, it inherits an action by $G$, and one immediately checks that this action is by isometries. Moreover, Genevois shows that these actions have features of acylindricity \cite{Genevois20b}. Under stronger assumptions, Murray, Qing, and Zalloum show that these spaces $Y_L$ stabilize.  

\begin{lemma} [\cite{MurrayQingZalloum20}]
Let $Y$ be a cocompact $\CAT$ cube complex with a factor system. Then there exists a constant $L$ such that any two disjoint hyperplanes $h, k$ in $Y$ are either $L$-well-separated or not well-separated. In other words, the metrics $d_L$ stabilize as $L \rightarrow \infty$. We call the smallest such $L$ the {\it separation constant} of $Y$.
\end{lemma}


The following result is due to Petyt, Spriano, and Zalloum in forthcoming work.  If $Y$ is a finite-dimensional $\CAT$ cube complex that has a factor system, and $G$ acts properly, cocompactly, and by cubical isometries on $Y$, then by \cite{AbbottBehrstockDurham21}, $G$ has a largest acylindrical action which we denote $G \curvearrowright \mathcal C(S)$.  (See the beginning of Section~\ref{subsec:ThecaseofHHGs} for a discussion of this notation.) 

\begin{theorem} [\cite{ZalloumPreprint}]
\label{thm:GenevoisABDequivalence}
Let $Y$ be a finite-dimensional $\CAT$ cube complex that has a factor system, and suppose $G$ acts properly, cocompactly, and by cubical isometries on $Y$.  If $L \in \mathbb{N}$ is sufficiently large such that $d_L$ has stabilised, then there exists a $G$-equivariant quasi-isometry $Y_L \rightarrow \mathcal C(S)$.
\end{theorem}

The following corollary is immediate.
\begin{corollary} \label{cor:LargestAcylindricityofGenevois}
Under the assumptions of Theorem \ref{thm:GenevoisABDequivalence}, the action of $G$ on $Y_L$ is a largest acylindrical action.
\end{corollary}

Thus Genevois' construction yields a largest acylindrical action provided that we restrict ourselves to groups acting geometrically on $\CAT$ cube complexes that have factor systems. Under these assumptions, we will also show that the hyperbolic projection $(G, Y, Y_L)$ is $\kappa$-injective for any $L$ greater than the separation constant. To do so, we first recall a result from \cite{MediciZalloum21}.

\begin{theorem} [{\cite[Corollary 6.11, Corollary 6.19]{MediciZalloum21}}]
\label{thm:kappageodesicsinject}
Let $Y$ be a cocompact $\CAT$ cube complex with a factor system, and let $L$ be greater than its separation constant.  There exists a constant $D \geq 0$ and a coarsely $G$-equivariant, coarsely Lipschitz map $Y \rightarrow Y_L$ that projects every unbounded sublinearly Morse geodesic to an unbounded $D$-quasi-ruler (and unparametrised quasi-geodesic) in $Y_L$. The induced map $\iota\colon \partial_{\kappa} Y \rightarrow \partial_{\infty} Y_L$ is injective for every sublinear function $\kappa$. Furthermore, if $\partial_{\kappa} Y$ is equipped with $\Vis$ and $\partial_{\infty} Y_L$ is equipped with the visual topology, then $\iota$ is continuous.
\end{theorem}

Using this result, we now show that  $(G,Y,Y_L)$ is a $\kappa$-injective hyperbolic projection.

\begin{lemma} \label{lem:Acylindricalactionincubulatedcaseiskappainjective}
Let $Y$ be a finite dimensional $\CAT$ cube complex that has a factor system, and suppose $G$ acts properly, cocompactly, and by cubical isometries on $Y$. Let $\kappa$ be a sublinear function and $L \in \mathbb{N}$ the separation constant of $Y$. Then $(G, Y, Y_L)$ is $\kappa$-injective, and the induced inclusion $\partial p \colon \partial_{\kappa} Y \hookrightarrow \partial_{\infty} Y_L$ coincides with the map $\iota$ from Theorem \ref{thm:kappageodesicsinject}.

\end{lemma}

\begin{proof}
Let $\gamma$ and $\gamma'$ be two $\kappa$-Morse quasi-geodesic rays in $Y$ that are $\kappa$-close. Without loss of generality, we may assume that $\gamma$ and $\gamma'$ have the same initial point, which we denote by $o$. To prove the lemma, we  must show the following three things: 
    \begin{enumerate}
        \item the projection $p \circ \gamma$ defines a point in $\partial_{\infty} Y_L$;
        \item the projections $p \circ \gamma$ and $p \circ \gamma'$ define the same point in $\partial_{\infty} Y_L$; and 
        \item the map that sends a family of $\kappa$-close $\kappa$-Morse quasi-geodesic rays to the point defined by their projections coincides with $\iota$.
    \end{enumerate}  
    Injectivity of $\iota$ will then imply that $p \circ \gamma$ and $p \circ \gamma'$ define the same point only if $\gamma$ and $\gamma'$ are $\kappa$-close, which finishes the proof of $\kappa$-injectivity.

We start by proving that $p \circ \gamma$ defines a point in $\partial_{\infty} Y_L$. By \cite[Theorem~3.13]{DurhamZalloum22}, there is  a chain of $L$-well-separated hyperplanes $h_l$ such that $\gamma$ crosses all $h_l$ (that is, $\gamma$ crosses each $h_l$ finitely many times and $h_l$ separates $o$ from an unbounded part of $\gamma$). The $h_l$ are ordered such that $h_l$ separates $h_{l-1}$ from $h_{l+1}$.

Let $(t_i)_i$ be an unbounded, increasing sequence in $\mathbb{R}$ and let $p_i := p \circ \gamma(t_i)$. We first show that $( p_i \mid p_j )_o \rightarrow \infty$. Let $D \geq 0$, choose $l_0 \geq D + L + 2$, and let $I$ be such that $h_{l_0}$ separates $o$ from $p_i$ for all $i \geq I$. Since the set underlying $Y_L$ consists of the vertices of $Y$, we can think of $p_i$ as a vertex in $Y$. Separation of $o$ and $p_i$ means separation as vertices in $Y$.

Fix $i,j\geq I$.  By definition, $d_L( p_i, p_j)$ is the size of a largest family of $L$-well-separated hyperplanes separating $p_i$ from $p_j$. Choose such a family of hyperplanes. This family can be decomposed into two subsets: the hyperplanes $\{ k_1, \dots, k_{l_1} \}$ that separate $p_i$ from $o$ and $p_j$, and the hyperplanes $\{ k'_1, \dots , k'_{l_2}\}$ that separate $p_i$ and $o$ from $p_j$. In particular, $d_L(p_i,p_j)=l_1+l_2$. Without loss of generality, we assume that the hyperplanes $k_l$ are ordered such that $k_l$ separates $o$ from $k_{l+1}$.  See Figure~\ref{fig:hps}.

\begin{figure}
    \centering
    \def\svgwidth{5in}
    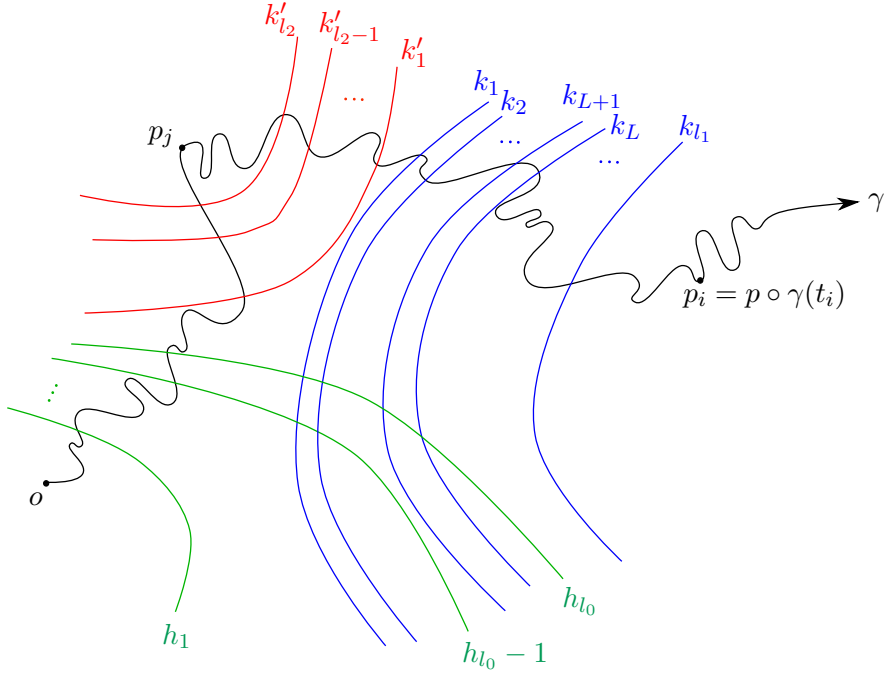
    \caption{The configuration of hyperplanes in the proof of Lemma~\ref{lem:Acylindricalactionincubulatedcaseiskappainjective}.}
    \label{fig:hps}
\end{figure}

Consider the family of hyperplanes $h_1, \dots, h_{l_0}, k_1, \dots, k_{l_1}$. We claim that the family $h_1, \dots, h_{l_0 - 1}, k_{L+2}, \dots, k_{l_1}$ is an $L$-well-separated family of hyperplanes separating $o$ from $p_i$. By the ordering on  the hyperplanes, it suffices to show that $h_{l_0 - 1}$ and  $k_{L+2}$ and disjoint and $L$-well-separated. We first show that $h_{l_0 - 1}$ and $k_{L+1}$ are disjoint, which immediately implies that $h_{l_0 - 1}$ and $k_{L+2}$ are disjoint, as well. If $h_{l_0 - 1}$ and $k_{L+1}$ intersect, then all the hyperplanes $k_1, \dots, k_L$ must intersect $h_{l_0 - 1}$, as well. Since $h_{l_0}$ separates $h_{l_0 - 1}$ from $p_i, p_j$ and all $k_l$ separate $p_i$ from $p_j$, it follows that  $k_1, \dots, k_{L+1}$ have to intersect both $h_{l_0 - 1}$ and $h_{l_0}$. Furthermore, the hyperplanes $k_1, \dots, k_{L+1}$ do not contain a facing triple, as they all separate $p_i$ from $p_j$. However, $h_{l_0 - 1}$ and $h_{l_0}$ are $L$-well-separated, which is a contradiction, and so $h_{l_0 - 1}$ and $k_{L+1}$ are disjoint.

We next prove that $h_{l_0 - 1}$ and $k_{L+2}$ are $L$-well-separated. Any hyperplane intersecting both $h_{l_0 - 1}$ and $k_{L+2}$ must also  intersect $k_{L+1}$. Thus, if $h_{l_0 - 1}$ and  $k_{L+2}$ are not $L$-well-separated, then neither are $k_{L+1}$ and $k_{L+2}$, which is a contradiction. Thus the family $h_1, \dots, h_{l_0 - 1}, k_{L+2}, \dots, k_{l_1}$ is a family of $L$-well-separated hyperplanes separating $o$ from $p_i$, and so $d_L(o, p_i) \geq l_0 - 1 + l_1 - (L+1)$.

An analogous argument shows that $d_L(o, p_j) \geq l_0 + l_2 - 2 - L$. Furthermore, since $d_L(p_i, p_j) = l_1 + l_2$, we have
\begin{equation*}
    \begin{split}
        (p_i \mid p_j)_o & = \frac{1}{2} ( d_L(o, p_i) + d_L(o, p_j) - d_L(p_i, p_j) )\\
        & \geq \frac{1}{2} ( 2l_0 + l_1 + l_2 - 4 - 2L - l_1 - l_2)\\
        & = l_0 - 2 - L \geq D.
    \end{split}
\end{equation*}
Therefore $\lim_{i, j \rightarrow \infty} (p_i \mid p_j)_o \geq D$ for all $D \geq 0$.

By an abuse of notation, we denote $o$ to be the projection of the the starting point of $\gamma$ under $p \colon Y \rightarrow Y_L$. By  \cite[Lemma~3.2]{DurhamZalloum22}, two $\kappa$-close quasi-geodesic rays cannot be separated by two $L$-well-separated hyperplanes. Therefore, if $(h_i)_i$ is a family of $L$-well-separated hyperplanes for $\gamma$, then $\gamma'$ crosses all $h_i$, as well. Now let $(t_i)_i, (s_j)_j$ be two unbounded, increasing sequences and $D > 0$. We can choose $l_0 \geq D + 2 + L$, $I_0$ such that for all $i, j \geq I_0$, $h_{l_0}$ separates $o$ from $p \circ \gamma(t_i)$ and $p \circ \gamma'(s_j)$.  The same argument as before for $p_i, p_j$ shows that $( p \circ \gamma(t_i) \vert p \circ \gamma'(s_j))_o \geq D$ for all $i, j \geq I_0$. Thus $( p \circ \gamma(t_i) \mid p \circ \gamma'(s_j))_o \xrightarrow{i, j \rightarrow \infty} \infty$, and so $p\circ \gamma$ and $p\circ \gamma'$ define the same point in $\partial_\infty Y_L$.

We have shown that any equivalence class of $\kappa$-Morse quasi-geodesic rays in $Y$ projects to a collection of paths in $Y_L$ that represents a unique point in $\partial_{\infty} Y_L$. Since every point in $\partial_{\kappa} Y$ can be represented by a geodesic ray, this map coincides with the map $\iota \colon \partial_{\kappa} Y \hookrightarrow \partial_{\infty} Y_L$ from Theorem \ref{thm:kappageodesicsinject}. In particular, this map is injective, and therefore $p$ is $\kappa$-injective.
\end{proof}

Given any cubulable group $G$ and any cubulation $G \curvearrowright Y$ with a factor system,  the largest acylindrical action on $Y_L$ is $\kappa$-injective and therefore the acylindrical topology, $\mathcal{T}(G, [Y_L])$, can be defined. The following lemma describes this topology.

\begin{lemma} \label{lem:Acylindricaltopologyequalsvisual}
If $L$ is greater than or equal to the separation constant, then the map $\partial p \colon (\partial_{\kappa} Y, \Vis) \hookrightarrow \partial_{\infty} Y_L$ is a homeomorphism onto its image. In particular, the topology on $\partial_{\kappa} Y$ inherited from the inclusion $\partial_{\kappa} Y \hookrightarrow \partial_{\infty} Y_L$ coincides with the topology inherited from the inclusion $\partial_{\kappa} Y \hookrightarrow \partial_{\infty} Y$.
\end{lemma}

\begin{proof}
Continuity of $\partial p$ is proven in \cite{MediciZalloum21} (see Theorem \ref{thm:kappageodesicsinject}). It remains to show that $\partial p$ sends open sets to open sets. Fix a sublinear function $\kappa$. By Theorem~\ref{thm:HYPVIS}, the topology $\Vis$ inherited from the canonical inclusion $\partial_{\kappa} Y \hookrightarrow \partial_{\infty} Y$ coincides with the topology $\Hyp$ (see Definition~\ref{def:HYP}). Furthermore, by  \cite[Lemma 4.1]{MediciZalloum21}, the family $\{ U_{o,k} \mid k \text{ hyperplane } \}$ is a basis for this topology, where $o$ is a fixed vertex in $Y$ and
\[ U_{o,k} := \{ \eta \in \partial_{\kappa} Y \mid \gamma_{\eta} \text{ crosses } k \}, \]
with $\gamma_{\eta}$ denoting the $\CAT$ geodesic ray starting at $o$ and representing $\eta$.

Fix $\xi \in \partial_{\kappa} Y$, and let $\gamma$ be the unique $\CAT$ geodesic from $o$ to $\xi$. Let $U_{o,k}$ be an open set that contains $\xi$. In order to show that the image of $U_{o,k}$ is open in $\partial p( \partial_{\kappa} Y)\subseteq \partial_\infty Y_L$ with respect to the subspace topology, we need to find a constant $R \geq 0$ such that the 
basic neighbourhood $V = \{ \eta \in \partial_{\infty} Y_L \mid ( \partial p(\xi) \mid \eta)_{p(o)} \geq R \}$ of $\partial p( \xi)$ satisfies $\partial p(\xi) \in V \cap \partial p( \partial_{\kappa} Y ) \subset \partial p(U_{o,k})$.

Since $\gamma$ is $\kappa$-Morse and $Y$ has separation constant $L$, there exists a constant $c > 0$ and a chain $(k_j)_j$ of $L$-well-separated hyperplanes such that $\gamma$ crosses $k_j$ at time $t_j$ with $t_j < t_{j+1}$ (cf. \cite[Theorem B]{MurrayQingZalloum20} or Theorem~\ref{thm:Morse iff contracting iff excursion}). By  \cite[Lemma 4.1]{MediciZalloum21} and its proof, there exists $J$ such that for all $j \geq J$, the hyperplane $k_j$ separates $k$ from $\xi$. In particular, for all $j \geq J$, the collection of $L$-well-separated hyperplanes separating  $k_j$  from $k$ has cardinality at least $j - J$.

Let $p \colon Y \rightarrow Y_L$ be a $G$-equivariant projection such that for every $y \in Y$, $p(y)$ is a vertex of the least-dimensional cube in $Y$ containing $y$. Observe that $y$ and $p(y)$, thought of as a vertex in $Y$, can be separated by at most one hyperplane $k_j$, as the $k_j$ are mutually disjoint. Suppose now that $y \in Y$ such that $k_j$ separates $o$ from $y$ for some $j \geq J + 2$. We conclude that for all $y'$ that lie on $k$ or on the same side of $k$ as $o$, we have $d_L( p(y'), p(y) ) \geq j - J - 2$.

Let $\delta$ denote the hyperbolicity constant of $Y_L$ and $D \geq 0$ a constant such that $p$ projects every geodesic ray in $Y$ to an unparametrised $D$-quasi-ruler. By Theorem \ref{thm:kappageodesicsinject}, such a $D$ exists and both $D$ and $\delta$ depend only on $L$. Let $j_0 := \lceil J + 2D + 2\delta + 3 \rceil$, and let $R := d_L( p(o), p \circ \gamma(t_{j_0}) )$. Define the set $V := \{ \eta \in \partial_{\infty} Y_L \mid ( \partial p(\xi) \mid \eta)_{p(o)} > R \}$. We claim that $\partial p(\xi) \in V \cap \partial p( \partial_{\kappa} Y ) \subset \partial p( U_{o,k} )$. By definition, $\partial p(\xi) \in V$. Now suppose $\eta \in V \cap \partial p( \partial_{\kappa} Y )$. Then there exists a unique $\CAT$ geodesic $\gamma'$ starting at $o$ such that $p \circ \gamma'$ represents $\eta$. By Theorem \ref{thm:kappageodesicsinject}, both $p \circ \gamma$ and $p \circ \gamma'$ are $D$-quasi-rulers. Let $y' \in \gamma'$ be the first point on $\gamma'$ such that $d_L(p(o), p(y') ) \geq R$. Since $p \circ \gamma'$ is an unparametrised $D$-quasi-ruler, this implies that $d_L(p(o), p(y')) \leq R + D$. Applying Lemma \ref{lem:ControllingdistanceviaGromovproduct} with $x' = p \circ \gamma(t_{j_0})$, we conclude that $d_L(p(y'), p \circ \gamma(t_{j_0}) ) \leq \vert R + D - R \vert + D + 2\delta = 2D + 2 \delta$. 
If $k$ separates $y'$ and $\gamma(t_{j_0})$, then by the reasoning in the previous paragraph setting $y=\gamma(t_{j_0})$, we have $d_L(p(y'),p\circ \gamma(t_{j_0}))\geq j_0-J-2\geq 2D+2\delta+1$, where the second inequality follows from our choice of $j_0$.  However, this is a contradiction, and so $k$ must separate $y'$ and $\gamma(t_{j_0})$.  By our choice of $t_{j_0}$, $\gamma(t_{j_0})$ does not lie on the same side of $k$ as $o$, and thus neither does $y'$.

This implies that $\gamma'$ contains points on both sides of $k$, namely $o$ and $y'$, and thus $\gamma'$ crosses $k$. Therefore $\eta\in U_{o,k}$, and so $V \cap \iota( \partial_{\kappa} Y ) \subset U_{o,k}$.  This proves that the inclusion $\partial p \colon \partial_{\kappa} Y \hookrightarrow \partial_{\infty} Y_L$ sends open sets to open sets and completes the proof of the lemma.
\end{proof}

\begin{corollary} \label{cor:VisualTopologyInvariance}
Let $Y, Y'$ be two $\CAT$ cube complexes that admit factor systems, and let $G$ be a group that acts geometrically and cubically on $Y$ and $ Y'$. The unique $G$-equivariant map $\partial_{\kappa} Y \rightarrow \partial_{\kappa} Y'$ is a homeomorphism with respect to the visual topology $\Vis$.
\end{corollary}

\begin{proof}
Let $X$ be the hyperbolic space obtained from $Y$ using Genevois' construction (Definition~\ref{def:GenevoisConstruction}) with the separation constant of $Y$, and let $X'$ be the space obtained from $Y'$ analogously. By Corollary \ref{cor:LargestAcylindricityofGenevois}, the actions of $G$ on $X$ and $X'$ are  acylindrical and largest, and thus $X \sim X'$. Therefore, the unique $G$-equivariant map $\partial_{\kappa} Y \rightarrow \partial_{\kappa} Y'$ is a homeomorphism with respect to the topologies induced by the inclusions $\partial_{\kappa} Y \hookrightarrow \partial_{\infty} X$ and $\partial_{\kappa} Y' \hookrightarrow \partial_{\infty} X'$, respectively. By Lemma \ref{lem:Acylindricaltopologyequalsvisual}, these topologies coincide with the visual topology on $\partial_{\kappa} Y$ and $\partial_{\kappa} Y'$ respectively.
\end{proof}

As outlined in the introduction, Corollary \ref{cor:VisualTopologyInvariance} extends to commensurable groups acting on cube complexes:

\begin{corollary} \label{cor:Invarianceundercommensurability}
Let $G, G'$ be two commensurable groups which act geometrically and cubically on $\CAT$ cube complexes $Y, Y'$ respectively, each of which admits a factor system. Then any shared finite index subgroup $H $ of $ G, G'$ induces an $H$-equivariant homeomorphism $\partial_{\kappa} Y \rightarrow \partial_{\kappa} Y'$ with respect to the visual topology.
\end{corollary}

\begin{remark}
In the cubulated case, we expect the cross ratio obtained in Remark \ref{rem:InducedMetric} to be as close to the cross ratio defined in \cite{CCM} as possible.
\end{remark}




\subsection{Geometric implications for equivariant maps} \label{subsec:GeometricImplication}

The continuity of the map $\partial_{\kappa} Y \rightarrow \partial_{\kappa} Y'$ has an immediate consequence for how badly $G$-equivariant quasi-isometries can deform geodesics.

\begin{proposition} \label{prop:ApproximationofMorsequasigeodesics}
Let $G, Y, $ and $Y'$ be as in Corollary \ref{cor:VisualTopologyInvariance}, let $F \colon Y \rightarrow Y'$ be a coarsely $G$-equivariant $(K, C)$-quasi-isometry, and let $f \colon \partial_{\kappa} Y \rightarrow \partial_{\kappa} Y'$ be the induced homeomorphism between boundaries. Fix a basepoint $o \in Y$, and set $o' := f(o)$. Fix $\xi \in \partial_{\kappa} Y$ and $\xi' := f(\xi)$, and let $\gamma$ be the unique geodesic from $o$ to $\xi$, and let $\gamma'$ the unique geodesic from $o'$ to $\xi'$.

For any $R', \epsilon' > 0$, there exist $R \geq R'$ and $0 < \epsilon \leq \epsilon'$, depending only on the data above and $D$, depending only on $\epsilon', K, C$, $\kappa(R')$, and the Morse gauge of $\gamma$, such that the following holds. If $\eta$ is a point in  $\partial_{\kappa} Y$ whose unique geodesic representative $\beta$ starting at $o$ satisfies $d_Y( \beta(R), \gamma(R)) \leq \epsilon$, then
\[ d_{Haus}( F \circ \beta\vert_{[0, R']}, \beta'\vert_{[0, R']} ) \leq D, \]
where $\beta'$ is the unique geodesic ray from $o'$ to $f(\eta)$.

\end{proposition}



In more classical proofs of continuity of the visual topology on boundaries (for example for geodesic hyperbolic spaces), one proves a result similar to the one above and concludes continuity of the boundary map (cf. \cite[Part III, Section H.3]{BH}). In our situation, this does not work, as  illustrated by Cashen's example \cite{Cashen} of quasi-isometric CAT(0) spaces with non-homeomorphic Morse boundaries w.r.t.\,visual topology. The issue arises when attempting to approximate the image of a $\kappa$-Morse geodesic ray under a quasi-isometry by a geodesic ray. While every point in the $\kappa$-Morse boundary of a proper geodesic metric space can be represented by a geodesic ray, the distance between this geodesic ray and a $\kappa$-close quasi-geodesic ray depends on the Morse gauge of the initial $\kappa$-Morse geodesic. Even if two $\kappa$-Morse geodesic rays stay close together for a long time, the Morse gauge of one cannot give us any information on the precision of approximation for the other, as can be seen in Cashen's example. Proposition \ref{prop:ApproximationofMorsequasigeodesics} shows that this issue does not occur when the quasi-isometry is $G$-equivariant and the spaces in questions are $\CAT$ cube complexes with factor systems. Somewhat unusually, this follows from continuity with respect to the visual topology, rather than the other way around.

\begin{proof}
Let $G, Y, Y', \kappa, F, f, o, o', \xi, \xi', \gamma, \gamma', R', \epsilon'$ be as in the statement of the proposition. Since $f$ is continuous with respect to the visual topology, there exist $R, \epsilon > 0$ such that if $\eta \in \partial_{\kappa} Y$ is  represented by a geodesic ray $\beta$ starting at $o$ such that $d_Y( \beta(R), \gamma(R)) \leq \epsilon$, then $d_{Y'}( \beta'(R'), \gamma'(R')) \leq \epsilon'$, where $\beta'$ is the unique geodesic ray from $o'$ to $f(\eta)$. Without loss of generality, we can choose $R \geq R'$ and $\epsilon \leq \epsilon'$. Since $Y$ and $Y'$ are $\CAT$ spaces, this implies that
\[ d_{Haus}( \beta\vert_{[0,R]}, \gamma\vert_{[0,R]}) \leq \epsilon \] and
\[ d_{Haus}( \beta'\vert_{[0,R']}, \gamma'\vert_{[0,R']}) \leq \epsilon'. \]
Since $\gamma$ is $\kappa$-Morse (with a particular, fixed Morse gauge), there exists a constant $D_1$, depending only on $\kappa(R')$, the Morse gauge of $\gamma$, and the quasi-isometry constants of $F$ such that
\[ d_{Haus}( \gamma'\vert_{[0, R']}, F \circ \gamma\vert_{[0, R']}) \leq D_1.  \]
Finally, since $F$ is a $(K, C)$-quasi-isometry, we have that
\[ d_{Haus}( F \circ \beta\vert_{[0, R]}, F \circ \gamma\vert_{[0, R]}) \leq K\epsilon + C \leq K \epsilon' + C. \]
Combining these three estimates, we obtain that
\[ d_{Haus}( \beta'\vert_{[0, \min(R, R')]}, F \circ \beta\vert_{[0, \min(R, R')]}) \leq \epsilon' + D_1 + K\epsilon' + C.\]
Setting $D := \epsilon' + D_1 + K\epsilon' + C$ proves the proposition.
\end{proof}




\section{Topology for groups with a coarse median structure} \label{sec:Topologyforgroupswithacoarsemedianstructure}




It can be difficult to check whether a given hyperbolic projection $(G, Y, X)$ is $\kappa$-injective, as one has to account for all $\kappa$-Morse quasi-geodesic rays in $Y$. Things become more manageable if we have some additional structure in the form of coarse medians. We first describe some basics about coarse median structures. In Section \ref{subsec:Definingthetopologyinthemediancase}, we define a topology on sublinearly Morse boundaries that takes into account a coarse median structure on the group $G$ and show that it has analogous properties to the topologies in the previous section. In Section \ref{subsec:ThecaseofHHGs}, we then give a concrete description of this topology in the case of hierarchically hyperbolic groups.


The definition of coarse medians goes back to Bowditch \cite{Bowditch13}. The following equivalent definition is introduced in \cite{NibloWrightZhang19}.

\begin{definition} \label{def:coarsemedian}
Let $Y$ be a metric space. A \textit{coarse median} on $Y$ is a map $\mu \colon Y^3 \rightarrow Y$ for which there exists a constant $C \geq 0$ such that for all $a, b, c, x \in Y$, we have
\begin{enumerate}
    \item $\mu(a,a,b) = a$, $\mu(a,b,c) = \mu(b,a,c) = \mu(b, c, a)$,
    
    \item $d( \mu( \mu(a,x,b),x,c) , \mu( a, x, \mu(b,x,c) ) \leq C$,
    
    \item $d( \mu(a, b, c), \mu(x, b, c) ) \leq Cd(a,x) + C$.

\end{enumerate}
\end{definition}


\begin{definition}
Two coarse medians $\mu, \mu' \colon Y^3 \rightarrow Y$ have {\it bounded distance} if there exists a constant $C \geq 0$ such that $d( \mu(a,b,c), \mu'(a,b,c) ) \leq C$ for all $a, b, c \in Y$. This defines an equivalence relation, and  an equivalence class $[\mu]$ of coarse medians is a {\it coarse median structure} on $Y$. The pair $(Y, [\mu])$ is a {\it coarse median space}.
\end{definition}

\begin{definition} \label{def:medianmap}
A map $\Psi \colon (Y, [\mu]) \rightarrow (Y', [\mu'])$ between coarse median spaces  is a {\it $D$-median map} if there is a constant $D\geq 0$ such that for all $a, b, c \in Y$, we have
\[ d( \Psi( \mu(a, b, c) ) , \mu'( \Psi(a), \Psi(b), \Psi(c) ) ) \leq D. \]
A  map is a {\it coarse median map} if it is a $D$-median map for some $D \geq 0$.
\end{definition}

It follows from the definition that the composition of two coarse median maps is again a coarse median map. 

\begin{definition} \label{def:coarsemediangroup}
A {\it coarse median group} is a pair $(G, [\mu])$, where $G$ is a finitely generated group and $[\mu]$ is a $G$-invariant coarse median structure on $G$ with respect to the word metric of a finite generating set.
\end{definition}

Given a quasi-isomery $\Psi \colon Y \rightarrow Y'$ with a quasi-inverse $\Psi^{-1}$ and a coarse median $\mu$ on $Y$, one can push-forward $\mu$ to a coarse median on $Y'$ defined by
\[ \Psi_*\mu(a,b,c) := \Psi( \mu( \Psi^{-1}(a), \Psi^{-1}(b), \Psi^{-1}(c) ) ) ). \]
Thus a coarse median structure on a finitely generated group can be thought of as a choice of coarse median structure on one of its geometric representations.  In particular, given a coarse median group $(G, [\mu])$, every geometric action $G \curvearrowright Y$ induces a $G$-invariant coarse median structure on $Y$. Abusing notation, we will frequently denote this coarse median structure by $[\mu]$ as well. Given a coarse median group $(G, [\mu])$ and two geometric actions $G \curvearrowright Y$ and $G\curvearrowright  Y'$, the induced coarsely $G$-equivariant quasi-isometry $\Psi \colon Y \rightarrow Y'$ is a coarse median map with respect to the coarse median structure $[\mu]$ on both $Y$ and $Y'$.

And interval $I \subset \mathbb{R}$ has a standard (coarse) median structure  that sends a triple of points $a, b, c \in I$ to the point in $\{ a, b, c\}$ that is in the middle of the three. We use this median to introduce the notion of a median quasi-geodesic.

\begin{definition}
A (quasi-)geodesic $\gamma \colon I \rightarrow Y$ in a coarse median space $(Y, [\mu])$  is a {\it median (quasi-)geodesic} if $\gamma$ is a coarse median map with respect to the median structure on $I$.
\end{definition}

It is straightforward to show that a median quasi-isometry sends median quasi-geodesics to median quasi-geodesics.

\subsection{Defining the topology in the coarse median case} \label{subsec:Definingthetopologyinthemediancase}

Given a geodesic metric space $Y$ with a coarse median $\mu$, we will  consider quasi-geodesic rays in $Y$ that are $\mu$-median and $\kappa$-Morse. In general, it is not obvious that every point $\xi \in \partial_{\kappa} Y$ admits a quasi-geodesic representative that is $\mu$-median. We say that $\partial_{\kappa} Y$ {\it has $\mu$-median representatives} if for every $\xi \in \partial_{\kappa} Y$ there exists a $\mu$-median quasi-geodesic ray representing $\xi$. By modifying the beginning of such a quasi-geodesic ray, one can ensure that such representatives start at any chosen basepoint in $Y$. Note that if  $\Psi \colon Y \rightarrow Y'$ is a quasi-isometry and $\mu$ is a coarse median on $Y$, then $\partial_{\kappa} Y$ has $\mu$-median representatives if and only if $\partial_{\kappa} Y'$ has $\Psi_*(\mu)$-median representatives. Thus we may say that the $\kappa$-boundary of  a group with a coarse median structure $[\mu]$  has $\mu$-median representatives if the $\kappa$-boundary of some (equivalently, any) space on which the group acts geometrically has this property.

If the $\kappa$-boundary of a pair $(G, [\mu])$ does not have $\mu$-median representatives, then the following results apply to a subset of $\partial_{\kappa} G$ consisting of all the $\kappa$-Morse points that admit a $\mu$-median representative. However, in all important examples for the results in this section, the entire $\kappa$-boundary has $\mu$-median representatives, which is why we simply assume this for the remainder of this section.

We now define the notion of a hyperbolic projection $(G,Y,X)$ being $(\kappa,\mu)$--injective.  The definition is almost the same as $\kappa$-injectivity, but we only consider $\kappa$-Morse quasigeodesic rays that are also $\mu$-median.

\begin{definition} \label{def:kappamuinjectivity}
Let $(G, [\mu])$ be a group with a coarse median structure and  $(G, Y, X)$  a hyperbolic projection. The $G$-equivariant quasi-isometry $G \to Y$ induces a coarse median structure on $Y$, which we also denote by $[\mu]$. The hyperbolic projection $(G, Y, X)$ is {\it $(\kappa, \mu)$-injective} if for every $\kappa$-Morse, $\mu$-median quasi-geodesic ray $\gamma$ in $Y$, the projection $p \circ \gamma$ defines a point in $\partial_{\infty} X$ and the projections of any two $\kappa$-Morse, $\mu$-median quasi-geodesic rays $\gamma, \gamma'$ in $Y$ define the same point in $\partial_{\infty} X$ if and only if $\gamma, \gamma'$ are $\kappa$-close.
\end{definition}

If a hyperbolic projection $(G, Y, X)$ is $(\kappa, \mu)$-injective and $\partial_{\kappa} G$ has $\mu$-median representatives, there is a well-defined map $\partial p \colon \partial_{\kappa} Y \hookrightarrow \partial_{\infty} X$ that sends any $\kappa$-Morse, $\mu$-median quasi-geodesic ray $\gamma$ in $Y$ to the point in $\partial_{\infty} X$ defined by $p \circ \gamma$. The pull-back of the visual topology on $\partial_{\infty} X$ induces a topology on $\partial_{\kappa} Y$ that we denote by $\mathcal{T}(G, [\mu], Y, X)$.

The following proposition is the analogue of Proposition~\ref{prop:CommutingDiagram} for $(\kappa,\mu)$-injective hyperbolic projections.
\begin{proposition} \label{prop:Commutingdiagrammediancase}
Let $(G, [\mu])$ be a group with a coarse median structure such that $\partial_{\kappa} G$ has $\mu$-median representatives, let $(G, Y, X)$ and $(G, Y', X')$ be two hyperbolic projections such that $X \sim X'$, and let $p \colon Y \rightarrow X$ and $p' \colon Y' \rightarrow X'$ be the induced projection maps. Let $\Psi \colon Y \rightarrow Y'$ and $\Phi \colon X \rightarrow X'$ be coarsely $G$-equivariant quasi-isometries induced by the geometric actions on $Y, Y'$ and the fact that $X \sim X'$. The following statements hold.
\begin{enumerate}
    \item $(G, Y, X)$ is $(\kappa, \mu)$-injective if and only if $(G, Y', X')$ is $(\kappa, \mu)$-injective.
    
    \item If $(G, Y, X)$ is $(\kappa, \mu)$-injective, then the following diagram commutes:
     \[
  \begin{tikzcd}
    \partial_{\kappa} Y \arrow{r}{\partial \Psi} \arrow[hook, swap]{d}{\partial p} & \partial_{\kappa} Y' \arrow[hook]{d}{\partial p'} \\
     \partial_{\infty} X \arrow{r}{\partial \Phi} & \partial_{\infty} X'
  \end{tikzcd}
\]

    \item The map
    \[ \partial \Psi \colon ( \partial_{\kappa} Y, \mathcal{T}(G, [\mu], Y, X)) \rightarrow (\partial_{\kappa} Y', \mathcal{T}(G, [\mu], Y', X') ) \]
    is a homeomorphism.
\end{enumerate}
\end{proposition}

\begin{proof}
As in the proof of Proposition \ref{prop:CommutingDiagram}, we obtain that the following diagram commutes up to bounded distance:
\[
  \begin{tikzcd}
    Y \arrow{r}{\Psi} \arrow[swap]{d}{p} & Y' \arrow{d}{p'} \\
     X \arrow{r}{\Phi} & X'
  \end{tikzcd}
\]
Suppose $(G, Y, X)$ is $(\kappa, \mu)$-injective, and let $\gamma'$ be a $\kappa$-Morse, $\mu$-median quasi-geodesic ray in $Y'$. Let $\overline{\Psi}$ denote a quasi-inverse of $\Psi$. Since the coarse median structures on $Y$ and $Y'$ are the ones induced by the action of $G$ and $\Psi$ is coarsely $G$-equivariant, we have that $\Psi$ and $\overline{\Psi}$ are median maps. In particular, $\overline{\Psi} \circ \gamma'$ is $\mu$-median. Since $(G, Y, X)$ is $(\kappa, \mu)$-injective, we conclude that $p \circ \overline{\Psi} \circ \gamma'$ determines a point in $\partial_{\infty} X$. Since $\Phi$ is a quasi-isometry between hyperbolic spaces with $D$-quasi-rulings, Corollary \ref{cor:quasiisometriesofquasiruledspacesextend} implies that $\Phi \circ p \circ \overline{\Psi} \circ \gamma'$ determines a point in $\partial_{\infty} X'$. Moreover, since $p' \circ \gamma'$ has finite Hausdorff distance from $\Phi \circ p \circ \overline{\Psi} \circ \gamma'$, it determines the same point in $\partial_{\infty} X'$. Therefore that $(G, Y', X')$ is $(\kappa, \mu)$-injective. Statements (2) and (3) then follow by the same argument as in Proposition \ref{prop:CommutingDiagram}.
\end{proof}




\subsection{The case of hierarchically hyperbolic groups} \label{subsec:ThecaseofHHGs}

Similar to the cubulated case for $\kappa$-injective hyperbolic projections, we want to consider $(\kappa, \mu)$-injective hyperbolic projections that arise from largest acylindrical actions. A natural setting in which we find such projections is hierarchically hyperbolic groups.  Hierarchical hyperbolicity provides an axiomatic framework that unifies the geometry of the mapping class group, 3--manifold groups, many cubulated groups, and many Artin groups.  As the precise definition of a hierarchically hyperbolic group is not necessary for this paper, we omit it and refer the interested reader to \cite{BehrstockHagenSisto2017a}.

Hierarchically hyperbolic groups admit a coarse median structure \cite{BehrstockHagenSisto2017a}, which we denote $[\mu]$.  Furthermore, every hierarchically hyperbolic group admits a largest acylindrical action (Theorem~\ref{thm:HHGLargest}).  This is the action of the group on the top-level curve graph in any \textit{maximized} structure (see \cite{AbbottBehrstockDurham21} for details on the construction of such a structure).  We denote the largest acylindrical action by $G \curvearrowright \mathcal{C}(S)$.\footnote{Technically, $\mathcal C(S)$ is used to denote the top-level curve graph in \textit{any} hierarchical structure on $G$.  As we will not change the hierarchical structure on $G$, this notation will not cause any confusion.}


This action is cobounded, so given any geodesic metric space $Y$ on which $G$ acts geometrically, $(G,Y,\mc C(S))$ is a hyperbolic projection.  Since $\mc C(S)$ is a hyperbolic space, it also admits a coarse median structure $[\mu_S]$.  Moreover, there is a coarsely Lipschitz projection map $\pi_S\colon Y\to \mc C(S)$ that is coarsely median; that this map is coarsely median follows immediately from the construction of a median structure on $G$ (see \cite{BehrstockHagenSisto2017a}).  The map $\pi_S$ plays the role of the map $p$ in the previous sections; here we use $\pi_S$ as that is the standard notation for such a map in a hierarchically hyperbolic group.  Durham and Zalloum \cite{DurhamZalloum22} show that the map $\pi_S$ induces a continuous injection $\partial_\kappa G\to \partial \mc C(S)$, where $\partial \mc C(S)$ has the Gromov topology.  This was first proven in the case of the Morse boundary of $G$ by the first author, Behrstock, and Durham \cite{AbbottBehrstockDurham21}.  Rephrasing these results in the language of this paper, we obtain: 

\begin{theorem} [\cite{AbbottBehrstockDurham21, DurhamZalloum22}]
Let $G$ be a hierarchically hyperbolic group, $\kappa$ any sublinear function, and $Y$ a geodesic metric space on which $G$ acts geometrically.  Then $(G, Y, \mathcal C(S))$ is $(\kappa, \mu)$-injective.
\end{theorem}

Combining this theorem with Proposition~\ref{prop:CommutingDiagramsublinearcase} shows that for any hierarchically hyperbolic group $G$, we have the acylindrical topology $\mathcal{T}( G, \mathcal{C}(S))$ on $\partial_\kappa G$.  In this section, we give an alternative characterization of this topology.

Let $G$ be a hierarchically hyperbolic group, and fix $o\in G$. The median rays in $G$ are exactly the \emph{hierarchy rays}, as defined by Behrstock, Hagen, and Sisto \cite{BehrstockHagenSisto2017a}; see also \cite{RussellSprianoTran}. In this paper, we will work only with median rays, instead of hierarchy rays.   The following result of Durham and Zalloum show that every point in the $\kappa$-Morse boundar of $G$ is represented by a median ray whose projection to $\mc C(S)$ is unbounded. 

\begin{lemma}[{\cite[Lemma~4.1 \& 6.3]{DurhamZalloum22}}]
For each $\xi\in \partial_\kappa G$, there is a $D$--median ray $h$ from $o$ to $\xi$ such that $\diam_S(\pi_S(h))=\infty$. 
\end{lemma}

We note that \cite[Lemma~4.1]{DurhamZalloum22} is about LQC spaces; by \cite[Theorem~F]{BehrstockHagenSisto2021}, hierarchically hyperbolic groups are LQC (in \cite{BehrstockHagenSisto2021} the LQC propery is called ``cubical approximation").  Additionally, \cite[Lemma~6.3]{DurhamZalloum22} is about weakly $\kappa$--Morse quasigeodesics;  this notion coincides with $\kappa$--Morse quasigeodesics by \cite[Theorem~D]{DurhamZalloum22}.

We are now ready to define a topology on $\partial_\kappa G$.  We first define a topology in terms of convergence and then give a description in terms of a neighborhood basis.

\begin{definition} \label{def:mediantopology}
Let $\xi_n,\xi\in \partial_\kappa G$, and let $h_n,h$ be $(D,D)$--median rays in $G$ representing $\xi_n$ and $\xi$, respectively, such that $\diam_S(h_n)=\diam_S(h)=\infty$.  The sequence  in   $\xi_n$ converges to $\xi$ in $\partial_M G$ if and only if 
\[
\liminf_{n\to\infty}\sup\{d_G(o,\mu(o,h_n(s),h(t)))\mid s,t\in \R_{\geq 0}\}=\infty.
\]
This defines a topology on $\partial_M G$, which we call the \emph{coarse median topology} and denote $\mc T_{med}$.
\end{definition}

For each $\zeta\in \partial _\kappa G$, let $h_\zeta$ denote a $D$--median ray from $o$ to $\zeta$ such that $\diam_S(\pi_S(h_\zeta))=\infty$.  For each $\xi\in \partial_M G$ and each $r\geq 0$, let 
\[
U_r(\xi):=\left\{\zeta\in \partial_M G \,\Bigg\vert \, \begin{aligned} \exists s_0,t_0&\in \mathbb R_{\geq 0} \textrm{ such that } \forall s\geq s_0,\forall t\geq t_0,  \\ &d_G(o,\mu(o, h_\zeta(s), h_\xi(t)))\geq r\end{aligned}\right\}.
\]

\begin{lemma} \label{lem:mediantopologybasis}
Fix $\xi\in \partial_\kappa G$.  The set $\{U_r(\xi)\mid r>0\}$ is a neighborhood basis of $\xi$ in $(\partial_\kappa G,\mc T_{med})$.  
\end{lemma}

\begin{proof}
We need to show that $\xi_n\to \xi$ in $(\partial_M G,\mc T_{med})$ if and only if for all $r>0$, there exists $N_r>0$ such that $\xi_n\in U_r(\xi)$ for all $n\geq N_r$.
Suppose first that for all $r>0$, there exists $N_r>0$ such that $\xi_n\in U_r(\xi)$ for all $n\geq N_r$. Fix $r$.  For all $n\geq N_r$, we have 
\[
d_G(o,\mu(o,h_n(s),h(t)))\geq r
\]
for all $s\geq s_0$ and $t\geq t_0$.  Taking $r\to \infty$, we see that 
\[
\liminf_{n\to\infty}\sup\{d_G(o,\mu(o, h_n(s),h(t)))\mid s,t\in \mathbb R_{\geq 0}\}=\infty.
\]
Thus $\xi_n\to\xi$ in $(\partial_M G,\mc T_{med})$.

For the other direction, suppose $\xi_n\not\in U_r(\xi)$.  Then $d_G(o,\mu(o,h_n(s),h(t)))\leq r$ for all $s,t\in \mathbb R_{\geq 0}$.  The definition of convergence then implies that $\xi_n\not\to \xi$ in $(\partial_M G,\mc T_{med})$. 
\end{proof}



Before proving the topologies $\mc T_{med}$ and $\mc T(G,\mc C(S))$ are equivalent, we 
recall that $\mathcal T(G,\mc C(S))$ is the pull-back of the visual topology on $\partial_{\infty} \mathcal{C}(S)$. In particular, a sequence $\xi_n \in \partial_{\kappa} G$ converges to $\xi$ if and only if $\partial \pi(\xi_n)$ converges to $\partial \pi(\xi)$ in $\partial_{\infty} \mathcal{C}(S)$, which occurs if and only if $(\partial \pi(\xi_n) \mid \partial \pi(\xi) )_o \xrightarrow{n \rightarrow \infty} \infty$.

\begin{theorem}
There is a homeomorphism $(\partial_\kappa G,\mc T_{med})$ to $(\partial_\kappa G,\mathcal T(G,\mathcal C(S)))$.
\end{theorem}

\begin{proof}
Since $(\partial_\kappa G,\mc T_{med})$ is sequential by construction, any sequentially continuous map $(\partial_\kappa G, \mc T_{med}) \to Y$ is continuous.  Similarly, $(\partial_\kappa G,\mc T(G,\mc C(S)))$ is first countable because the Gromov topology on $\mc C(S)$ is, and so any sequentially continuous map $(\partial_\kappa G, \mc T(G,\mc C(S))) \to Y$ is continuous.

We will show that the identity map $(\partial_\kappa G, \mc T_{med}) \to (\partial_\kappa G,\mc T(G,\mc C(S)))$ is a homeomorphism.  Since the map is clearly bijective, it suffices to show that for any $\xi_n,\xi\in \partial_\kappa G$, we have 
$\xi_n\to \xi$ in $\mc T_{med}$ if and only if $\xi_n\to\xi$ in $\mc T(G,\mc C(S))$.   

Let $h_n,h$ be $D$--median rays representing $\xi_n$ and $\xi$, respectively, such that $\diam_S(h_n)=\diam_S(h)=\infty$.

By definition, $\xi_n\to\xi$ in $\mc T_{med}$ if and only if 
\[
\liminf_{n\to\infty}\sup\{d_G(o,\mu(o,h_n(s),h(t)))\mid  s,t\in \R_{\geq 0}\}=\infty.
\]
The argument of \cite[Lemma~6.4]{DurhamZalloum22} shows that this happens if and only if 
\begin{equation} \label{eqn:projofmed}
\liminf_{n\to\infty}\sup\{d_S(\pi_S(o), \pi_S(\mu(o,h_n(s),h(t))))\mid s,t\in \R_{\geq 0}\}=\infty.
\end{equation}
Note that this is a stronger statement than \cite[Corollary~6.5]{DurhamZalloum22}. 
Since $\pi_S$ is a coarse median map, this is equivalent to 
\[
\liminf_{n\to\infty}\sup\{d_S(\pi_S(o), \mu(\pi_S(o),\pi_S(h_n(s)),\pi_S(h(t))))\mid  s,t\in \R_{\geq 0}\}=\infty.
\]
Moreover, there is a constant $A$ depending only on the hyperbolicity constant of $\mc C(S)$ such that for any  points $a,b\in \mc C(S)$,
\[
\left\vert d_S(o,\mu_S(o,a,b)) - (a\mid b)_{o}\right\vert\leq A.
\]
 Therefore \eqref{eqn:projofmed} is equivalent to 
\[
\liminf_{n\to\infty}\sup\{(h_n(s)\mid h(t))_{o}\mid s,t\in \mathbb R_{\geq 0}\}=\infty.
\]
Finally, this is true if and only if $\partial \pi(\xi_n)$ converges to $\partial \pi(\xi)$ in $\partial\mc C(S)$, or
equivalently, $\xi_n\to\xi$ in $(\partial_\kappa G,\mc T(G,\mc C(S)))$.
\end{proof}

We end with an example that shows the coarse median topology is not equivalent to the sublinear topology.

\begin{example}
Consider the right-angled Artin group $A_{CK} := \langle a, b, c, d \mid [a,b] = [b,c] = [c,d] = 1 \rangle$; this group is known as the Croke--Kleiner group \cite{CrokeKleiner00}. Let $Y$ denote the universal covering of its Salvetti complex, which is a $2$-dimensional $\CAT$ cube complex. In \cite[Proposition~3.3]{Medici19} (see, in particular, \cite[Figure~4]{Medici19}), the second author constructed a Morse geodesic ray $\gamma$ with the following property. For any $n \in \mathbb{N}$ there exists a Morse $(8 \sqrt{2}, 1)$-quasi-geodesic ray $\gamma_n$ such that the first $n$ hyperplanes crossed by $\gamma$ are also crossed by $\gamma_n$ (in the sense that $\gamma_n$ crosses each of these hyperplanes exactly once). Furthermore, $\gamma_n$ can be constructed so that when $\kappa \equiv 1$, the intersection of $\gamma_n$  with the $\kappa$-neighbourhood $\mathcal{N}_{\kappa}(\gamma, 1)$ for  is contained in the ball $B_{1}(\gamma(0))$. 

Now fix any $\kappa$. A minor modification of this construction ensures that the ray $\gamma_n$ can be chosen so that its intersection with $\mathcal{N}_{\kappa}(\gamma, 1)$ is contained in the ball $B_{\kappa(0)}(\gamma(0))$. The quasi-geodesic rays $\gamma_n$ represent a sequence of points $\xi_n \in \partial_{\kappa} Y$, and $\gamma$ represents a point $\xi \in \partial_{\kappa} Y$. Since the $\gamma_n$ only intersect $\mathcal{N}_{\kappa}(\gamma, 1)$ in a fixed ball around $\gamma(0)$, the points $\xi_n$ do not converge to $\xi$ in the sublinear topology. However,  there are at least $n$  hyperplanes separating $\gamma(0)$ from both $\xi_n$ and $\xi$.  Thus the distance in $A_{CK}$ from $\gamma(0)$ to the median $\mu( \gamma(0), \xi_n, \xi)$ diverges to infinity, implying that $\xi_n \rightarrow \xi$ in $\mathcal{T}_{med}$.
\end{example}




\section{Morse boundaries of graph manifolds} \label{sec:MorseboundariesofRAAGsandgraphmanifolds}

In this section we provide an explicit description of the topology in the case of a particular type of hyperbolic projection of non-positively curved graph manifolds. A non-positively curved graph manifold $M$ can be decomposed into finitely many compact, non-positively curved manifolds $M_1, \dots, M_n$, each of which has totally geodesic boundary and a Seifert fibration such that the metric on each $M_i$ has a local product structure compatible with the fibration (cf. \cite{CrokeKleiner02}). Without loss of generality, our choice of decomposition is minimal, that is,  there is no collection of  $M_i$ whose union is connected and admits a Seifert fibration. We call the $M_i$ the {\it pieces} of $M$. Furthermore,  the universal cover $X$ of $M$ can be decomposed into the lifts of the various $M_i$; each such lift is called a \textit{piece} of $X$. Each piece of $X$ is a product of the universal covering of a compact $2$-orbifold with boundary with $\mathbb{R}$. Let $G$ denote the fundamental group of $M$. Fundamental groups of non-positively curved graph manifolds are the motivating example for the definition of CK-admissible groups (see \cite{CrokeKleiner02}). 

The decomposition of a graph manifold provides a description of $G$ as a graph of groups, and $G$ acts coboundedly and by isometries on the Bass-Serre tree of this decomposition. By \cite[Corollary 5]{HagenRussellSistoSpriano22}, the action of $G$ on its Bass-Serre tree is the largest acylindrical action of $G$. We will use this fact to fully describe the topology on $\partial_{\kappa}X$ induced by this largest acylindrical action of $G$.  The case of graph  manifolds is actually a special case of a more general statement about $\kappa$-injective projections to trees; see Theorem~\ref{thm:MorseboundarycharacterisationasCantorchain}.

In light of this, we begin with some general results about trees.  All trees in this section are assumed to have edges of length $1$. The \textit{degree} (or \textit{valence}) of a vertex is the number of incident edges.  Any two trees with the property that every  vertex has countably infinitely degree are isomorphic.  We denote this tree of countably infinite valence by $T_{\infty}$.

Consider a tree $T_2$ such that all its vertices have finite degree at least three, and let $T_1 \subset T_2$ be a subtree whose vertices all have degree at least three as well. The tree $T_1$ is {\it entwined in $T_2$} if every geodesic ray in $T_1$ contains a vertex whose degree in $T_2$ is strictly larger than in $T_1$. Equivalently, $T_1$ is entwined in $T_2$ if and only if the visual boundary of $T_1$ has empty interior when considered as a subset of the visual boundary of $T_2$ (see \cite{CharneyCordesSisto} for the origin of this definition).

Let $T_1 \subset T_2 \subset \dots \subset T_{\infty}$ be an infinite, nested family of locally finite trees inside $T_{\infty}$. The tree $T_i$ is {\it strongly entwined} in $T_{i+1}$ if for every vertex in $T_i$, its degree in $T_{i+1}$ is strictly greater than in $T_i$.
The visual boundaries $\partial_{\infty} T_i$ form a nested sequence of subspaces in the visual boundary $\partial_{\infty} T_{\infty}$. In particular, their union $\bigcup_{i = 1}^{\infty} \partial_{\infty} T_i$ forms a subspace of $\partial_{\infty} T_{\infty}$.
A sequence of strongly entwined subtrees $T_1 \subset T_2 \subset \dots \subset T_{\infty}$ {\it fills $T_{\infty}$} if every edge in $T_{\infty}$ is contained in some $T_i$, i.e., $\bigcup_{i=1}^{\infty} T_i = T_{\infty}$.  

The following lemma is the main technical result needed to describe the acylindrical topology $\mc T(G, T_\infty)$ on $\partial_\kappa G$.  We defer its proof to Section~\ref{subsec:UniquenessofCantorchain}.

\begin{lemma} \label{lem:UniquenessofCantorchain}
Let $T_1 \subset T_2 \subset \dots \subset T_{\infty}$ and $T'_1 \subset T'_2 \subset \dots \subset T_{\infty}$ be two families of strongly entwined subtrees of $T_{\infty}$  that each  fill $T_{\infty}$.  There exists a homeomorphism $\phi \colon \partial_{\infty} T_{\infty} \rightarrow \partial_{\infty} T_{\infty}$  that restricts to a homeomorphism $\bigcup_{i=1}^{\infty} \partial_{\infty} T_i \rightarrow \bigcup_{i=1}^{\infty} \partial_{\infty} T'_i$.

In other words, the subspace of $\partial_{\infty} T_{\infty}$ defined by a filling sequence of strongly entwined trees is unique up to homeomorphism, and this homeomorphism extends to $\partial_{\infty} T_{\infty}$. 
\end{lemma}

We call any topological space homeomorphic to a subspace of $\partial_\infty T_\infty$ defined by a filling chain of strongly entwined trees a {\it Cantor chain}.

The $\kappa$-Morse boundary of a finitely generated group $G$ with identity element $e$ can be described as the union over all Morse gauges $N$ of sets $\partial^N_{e,\kappa} G$, each of which consists of all of equivalence classes of $\kappa$-Morse quasi-geodesic rays based at $e$  with Morse gauge $N$. The following theorem follows immediately from Lemma~\ref{lem:UniquenessofCantorchain}.

\begin{theorem} \label{thm:MorseboundarycharacterisationasCantorchain}
Let $G$ be a finitely generated group that acts coboundedly on $T_{\infty}$ such that $(G, T_{\infty})$ is $\kappa$-injective for a sublinear function $\kappa$, and let $\partial p \colon \partial_M G \hookrightarrow \partial_{\infty} T_{\infty}$ be the unique $G$-equivariant inclusion. For every Morse gauge $N$, let $T_{N,\kappa}$
denote the convex hull of $\partial p( \partial_{e, \kappa}^N G )$ in $T_\infty$; this is a subtree of $T_{\infty}$.

Suppose there is a sequence of Morse gauges $N_1 \leq N_2 \leq \dots$ such that the sequence $T_{N_1,\kappa} \subset T_{N_2,\kappa} \subset \dots$ is strongly entwined and filling, and such that $\bigcup_{i=1}^{\infty} \partial_{e, \kappa}^{N_i} G = \partial_{\kappa} G$. Then $\partial_{\kappa} G$ equipped with $\mathcal{T}(G, T_{\infty})$ is a Cantor chain.
\end{theorem}

Theorem~\ref{thm:MorseboundarycharacterisationasCantorchain} applies, for example, to suitable graphs of groups whose action on their Bass-Serre-tree is largest acylindrical and which are $\kappa$-injective. This includes graph manifolds, which yields the following corollary.  This result should be compared with  \cite[Corollary 5(4)]{HagenRussellSistoSpriano22}, where the Morse boundary of graph manifolds, equipped with a different topology, is described. The  phenomena that occur are quite similar with the notable difference that Cantor chains are embedded in $\partial_{\infty} T_{\infty}$, whereas $\omega$-Cantor chains are not.

\begin{corollary} \label{cor:Topologyforgraphmanifolds}
Let $G$ be the fundamental group of a non-positively curved graph manifold that consists of at least two pieces. Then $\partial_{\kappa} G$ equipped with the acylindrical topology is a Cantor chain.
\end{corollary}

\begin{proof}[Sketch of proof ]
Let $M$ be a non-positively curved graph manifold, $X$ its universal cover, and $G = \pi_1(M)$. Since $M$ is a graph manifold, it consists of finitely many pieces, each of which is a compact circle bundle over a compact 2-orbifold. Let $D \geq 0$ be such that for any two points in two boundary components of a piece, their distance in the piece is at most $D$. The decomposition of $M$ into pieces yields a decomposition of $X$ into pieces, which are separated by flat copies of $\mathbb{R}^2$ in $X$. To each bounding torus in the decomposition of $M$, we associate the family of bounding flats in $X$ that are lifts of this torus. Two bounding flats are \emph{of the same type} if they are lifts of the same bounding torus. By our choice of $D$, for each point $p$ in $X$ that lies in a flat bounding a piece, there exists a flat of every other type bounding the same piece at distance at most $D$ from  $p$.

Since $M$ is a graph manifold, the action of $G$ on  its Bass-Serre tree $T_{\infty}$ is a largest acylindrical action. This action corresponds to a projection from $X$ to $T_{\infty}$ which maps every piece of $X$ to its corresponding vertex in the Bass-Serre tree.

Since $X$ is non-positively curved, the Morse gauge of a $\kappa$-Morse quasi-geodesic ray corresponds to a constant $C$ such that the quasi-geodesic ray is $\kappa$-contracting with respect to the constant $C$ (see \cite{QingRafiTiozzo}). Our goal is to find constants $C_1 \leq C_2 \leq \dots$ such that, using the notation of Theorem \ref{thm:MorseboundarycharacterisationasCantorchain}, the family of trees $( T_{C_i, \kappa} )_i$ is strongly entwined and filling. Fix a basepoint $o \in X$, which corresponds to a basepoint $\overline{o} \in T_{\infty}$. Consider a geodesic ray $\gamma$ in $X$ that starts at $o$. In order to be $\kappa$-contracting with respect to $C_i$, the time $\gamma$ spends in every piece of $X$ has to be bounded in terms of $C_i \kappa(d(o, \gamma(t))$, that is, the time spent in each piece can grow at most sublinearly in the distance to the starting point. In particular, if $\gamma$ enters a piece $P$, the distance from the entrance point to $o$ determines a finite set of flats bounding $P$, and $\gamma$ has to leave $P$ through one of these flats. Thus,  projecting the set of all  geodesic rays based at $o$ that are $\kappa$-contracting with respect to $C_i$  to $T_{\infty}$ results in a locally finite tree.  Call this tree $T_{C_i, \kappa}$.

Since  $\gamma$ begins at $o$, it can only enter each piece $P$ of $X$  through a specific bounding flat $F$. Moreover, since $\gamma$ is $\kappa$-contracting with respect to $C_i$, there exist a collection of bounding flats that are the farthest away from $F$ through which $\gamma$ can exit.  For each point $p$ on such a farthest flat $F'$, there exists a flat $F'(p)$ bounding $P$ such that $d(p, F'(p)) \leq D$, and such that if  $\gamma$ is $\kappa$-contracting with respect to $C_i$, then it cannot leave $P$ through $F(p)$.

In order for $T_{C_{i}, \kappa}$ to be entwined in $T_{C_{i+1}, \kappa}$, we need to make sure that  every piece in $X$ has a bounding flat with the property that: (1) no  geodesics based at $o$ that are $\kappa$-contracting with respect to $C_i$  cross this bounding flat; but (2) there exists a  geodesic based at $o$ that is $\kappa$-contracting with respect to $C_{i+1}$  and does cross this bounding flat. We can accomplish this by increasing $C_i$ enough so that $\gamma$ is allowed to spend $D$ more time in every piece in $X$. By the reasoning in the previous paragraph, this increase will allow $\gamma$ to exit through a new flat (the flat $F'(p)$).  
The contracting property and properties of closest projection in graph manifolds ensures that it is always possible to choose such a $C_{i+1}$.  Thus  
we obtain a sequence of strongly entwined trees $T_{C_i, \kappa}$. These trees are filling, as any edge in $T_{\infty}$ is part of the projection of some $\kappa$-contracting geodesic ray for sufficiently large constant $C$. By Theorem \ref{thm:MorseboundarycharacterisationasCantorchain}, the inclusion $\partial_{\kappa}X \hookrightarrow \partial_{\infty} T_{\infty}$ induces the topology of a Cantor chain on $\partial_{\infty} X$.
\end{proof}

\subsection{Proof of Lemma \ref{lem:UniquenessofCantorchain}}\label{subsec:UniquenessofCantorchain}

We first fix the notation and terminology that will be used throughout the proof.  Given a tree $T$, we can fix a vertex $v^0$ to turn this into a rooted tree. We can picture this tree as a `downwards oriented family tree' by placing the vertex $v^0$ at the top and orienting all edges of $T$ to point away from $v^0$. If there is an oriented edge from a vertex $v$ to a vertex $w$, we say $w$ is a {\it child} of $v$. A vertex $w$ is a {\it descendant} of $v$ if there exists an oriented path from $v$ to $w$. The vertex $v^0$ is {\it generation zero}, its children are {\it generation one}, and a vertex $v$ is in generation $d(v^0,v)$, where $d$ denotes the distance in the tree.

Let $(T, v^0)$ be a rooted tree as above. Let $S \subset T$ be a subtree and $v$  a vertex in $S$.  A child $w$ of $v$ in $T$ is an {\it occupied child of $v$ with respect to $S$} if $w$ is in $S$. Otherwise, we call $w$  an {\it unoccupied child of $v$ with respect to $S$}. A finite subtree $S$ of $T$ containing $v^0$  is  an {\it auxiliary subtree} if:
\begin{enumerate}
    
    
    \item all children of $v^0$ are occupied with respect to $S$, and 
    
    \item every vertex in $S$ except $v^0$ has at least one unoccupied child. We call these vertices the {\it open ends of $S$}.
\end{enumerate}

Let $T_1 \subset T_2 \subset \dots$ and $T'_1 \subset T'_2 \subset \dots$ be two strongly entwined, filling families of subtrees of $T_{\infty}$. We will inductively construct a map from the set of vertices of $T_{\infty}$ to itself that induces the desired homeomorphism on the boundary. For ease of notation, we denote the set of vertices of a tree $T$ by $V(T)$.

\subsubsection*{Step 1: Constructing a map $\Phi\colon V(T_1) \rightarrow V(T'_1)$}

\begin{figure}
 
  \begin{tikzpicture}

\draw[orange] (0,0) -- (-2, -1);
\draw[orange] (0,0) -- (-1, -1);
\draw[orange] (0,0) -- (0, -1);
\draw[orange] (0,0) -- (1, -1);
\draw[orange] (0,0) -- (2, -1);

\draw[orange] (5, 0) -- (4, -1);
\draw[orange] (5, 0) -- (5, -1);
\draw[orange] (5, 0) -- (6, -1);
\draw[orange] (6, -1) -- (7, -2);
\draw[orange] (7, -2) -- (8, -3);

\draw[green] (-2, -1) -- (-2.25, -2);
\draw[green] (-2, -1) -- (-1.75, -2);
\draw[green] (-2.25, -2) -- (-2.5, -3);
\draw[fill, green] (-2.25, -2) circle (1pt);

\draw[dashed] (-2.25, -2) -- (-2.1, -3);
\draw[dashed] (-1.75, -2) -- (-1.9, -3);
\draw[dashed] (-1.75, -2) -- (-1.6, -3);
\draw[dashed] (-2.5, -3) -- (-2.6, -4);
\draw[dashed] (-2.5, -3) -- (-2.4, -4);

\draw[green] (4, -1) -- (3.75, -2);
\draw[green] (4, -1) -- (4, -2);
\draw[green] (4, -1) -- (4.25, -2);

\draw[dashed] (3.75, -2) -- (3.65, -3);
\draw[dashed] (3.75, -2) -- (3.85, -3);
\draw[dashed] (4, -2) -- (3.9, -3);
\draw[dashed] (4, -2) -- (4.1, -3);
\draw[dashed] (4.25, -2) -- (4.15, -3);
\draw[dashed] (4.25, -2) -- (4.35, -3);

\draw[blue] (-1, -1) -- (-1.25, -2);
\draw[blue] (-1, -1) -- (-1, -2);
\draw[blue] (-1, -1) -- (-0.75, -2);

\draw[dashed] (-1.25, -2) -- (-1.35, -3);
\draw[dashed] (-1.25, -2) -- (-1.15, -3);
\draw[dashed] (-1, -2) -- (-1.1, -3);
\draw[dashed] (-1, -2) -- (-0.9, -3);
\draw[dashed] (-0.75, -2) -- (-0.85, -3);
\draw[dashed] (-0.75, -2) -- (-0.65, -3);

\draw[blue] (5, -1) -- (4.75, -2);
\draw[blue] (5, -1) -- (5, -2);
\draw[blue] (5, -1) -- (5.25, -2);

\draw[dashed] (5.25, -2) -- (5.35, -3);
\draw[dashed] (5.25, -2) -- (5.15, -3);
\draw[dashed] (5, -2) -- (5.1, -3);
\draw[dashed] (5, -2) -- (4.9, -3);
\draw[dashed] (4.75, -2) -- (4.85, -3);
\draw[dashed] (4.75, -2) -- (4.65, -3);

\draw[teal] (0, -1) -- (-0.25, -2);
\draw[teal] (0, -1) -- (0.25, -2);

\draw[dashed] (-0.25, -2) -- (-0.35, -3);
\draw[dashed] (-0.25, -2) -- (-0.15, -3);
\draw[dashed] (0.25, -2) -- (0.35, -3);
\draw[dashed] (0.25, -2) -- (0.15, -3);

\draw[teal] (6, -1) -- (5.75, -2);
\draw[teal] (5.75, -2) -- (5.5, -3);
\draw[fill, teal] (5.75, -2) circle (1pt);

\draw[dashed] (5.75, -2) -- (5.85, -3);
\draw[dashed] (5.5, -3) -- (5.6, -4);
\draw[dashed] (5.5, -3) -- (5.4, -4);

\draw[purple] (1, -1) -- (0.75, -2);
\draw[purple] (1, -1) -- (0.9, -2);
\draw[purple] (1, -1) -- (1.1, -2);
\draw[purple] (1, -1) -- (1.25, -2);

\draw[dashed] (0.75, -2) -- (0.7, -3);
\draw[dashed] (0.75, -2) -- (0.8, -3);
\draw[dashed] (0.9, -2) -- (0.85, -3);
\draw[dashed] (0.9, -2) -- (0.95, -3);
\draw[dashed] (1.1, -2) -- (1.05, -3);
\draw[dashed] (1.1, -2) -- (1.15, -3);
\draw[dashed] (1.25, -2) -- (1.2, -3);
\draw[dashed] (1.25, -2) -- (1.3, -3);

\draw[purple] (7, -2) -- (6.5, -3);
\draw[purple] (7, -2) -- (6.75, -3);
\draw[purple] (7, -2) -- (7, -3);
\draw[purple] (6.5, -3) -- (6.25, -4);
\draw[fill, purple] (6.5, -3) circle (1pt);

\draw[dashed] (6.5, -3) -- (6.55, -4);
\draw[dashed] (6.75, -3) -- (6.65, -4);
\draw[dashed] (6.75, -3) -- (6.85, -4);
\draw[dashed] (7, -3) -- (6.9, -4);
\draw[dashed] (7, -3) -- (7.1, -4);
\draw[dashed] (6.25, -4) -- (6.2, -5);
\draw[dashed] (6.25, -4) -- (6.3, -5);

\draw[red] (2, -1) -- (1.75, -2);
\draw[red] (2, -1) -- (2.25, -2);
\draw[red] (2.25, -2) -- (2.5, -3);
\draw[fill, red] (2.25, -2) circle (1pt);

\draw[dashed] (1.75, -2) -- (1.65, -3);
\draw[dashed] (1.75, -2) -- (1.85, -3);
\draw[dashed] (2.25, -2) -- (2, -3);
\draw[dashed] (2.25, -2) -- (2.25, -3);
\draw[dashed] (2.5, -3) -- (2.4, -4);
\draw[dashed] (2.5, -3) -- (2.6, -4);

\draw[red] (8, -3) -- (7.75, -4);
\draw[red] (8, -3) -- (8, -4);
\draw[red] (8, -3) -- (8.25, -4);

\draw[dashed] (7.75, -4) -- (7.65, -5);
\draw[dashed] (7.75, -4) -- (7.85, -5);
\draw[dashed] (8, -4) -- (7.9, -5);
\draw[dashed] (8, -4) -- (8.1, -5);
\draw[dashed] (8.25, -4) -- (8.15, -5);
\draw[dashed] (8.25, -4) -- (8.35, -5);

\end{tikzpicture}
    \caption{Construction of the map in Step 1. We first send the orange subtree to the orange subtree. We then send green to green, blue to blue, etc. We proceed inductively to the next (dashed) layer. Observe that for every colour, the corresponding subtree on the left has equally many vertices with unoccupied children as the corresponding subtree on the right.}
    \label{fig:treemap}
\end{figure}
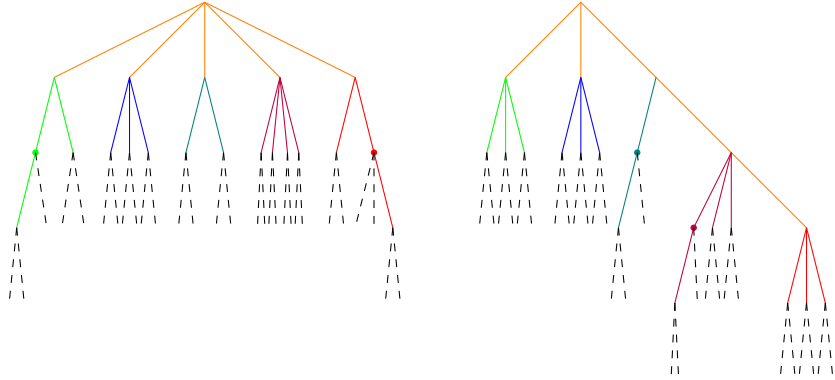

The construction of this map will also be inductive. Choose vertices $v^0 \in T_1$ and $v'^0 \in T'_1$ and consider $T_1$ and $T'_1$ as rooted trees with $v^0$ and $v'^0$ being their respective generation zero. Define $\Phi(v^0) := \Phi(v'^0)$. There are now  three cases to consider, depending on the respective degrees of $v^0$ and $v'^0$.

If $\deg(v^0) = \deg(v'^0)$,  label the children of $v^0$ by $a_1, \dots, a_N$, the children of $v'^0$ by $b_1, \dots, b_N$, and define $\Phi( a_n ) := b_n$.

If $\deg(v^0) > \deg(v'^0)$, let $S$ to be the (finite) subtree of $T_1$ which contains $v^0$ and all its children. Note that $S$ is an auxiliary subtree and it has $\deg(v^0)$ many open ends. Choose an auxiliary subtree $S'$ of $T'_1$ which also has $\deg(v^0)$ many open ends. Since $\deg(v^0) \geq \deg(v'^0)$, such an auxiliary subtree exists. Choose a bijection $\sigma$ from the open ends in $S$ to the open ends in $S'$, and define $\Phi := \sigma$ on these open ends.

If $\deg(v^0) < \deg(v'^0)$,  define $S'$ to be the subtree of $T'_1$ which contains $v'^0$ and all its children and choose an auxiliary subtree $S$ of $T_1$ which has $\deg(v'^0)$ many open ends. Note that $S$  contains all vertices in generation one of $T_1$, by property (2) of the definition of an auxiliary subtree.  As before,  choose a bijection $\sigma$ from the open ends of $S$ to the open ends of $S'$, and define $\Phi := \sigma$ on the set of vertices that appear as open ends in $S$. 

So far, in all cases we have defined $\Phi$ on $v^0$ and all its descendants that appear in the auxiliary subtree $S$.  (If $\deg(v^0) = \deg(v'^0)$, we can take $S$ to be $v^0$ together with its children.) In order to extend $\Phi$ to all vertices of $T_1$, we cut $T_1$ into several new rooted trees. Namely, for every open end $v$ of $S$, define a tree $T_v$ whose vertices are $v$ and all its descendants that do not require passing through an occupied child. This tree is rooted at $v$, and $\Phi$ is not defined for any vertex in $T_v$ except  the root. 
In the same way, we obtain trees $T_{v'}$ for every open end $v'$ of $S'$ in $T'_1$. The definition of $\Phi$ so far provides a bijection between the collection of trees $\{ T_v \mid v \text{ open end of } S\}$ and $\{ T_{v'} \mid v' \text{ open end of } S'\}$. Repeat the construction above on each $T_v$ to obtain auxiliary subtrees of $T_v$ and $T_{\Phi(v)}$, and extend the map $\Phi$ to the vertices contained in these auxiliary subtrees. Continuing inductively, we obtain a map $\Phi\colon V(T_1) \rightarrow V(T'_1)$ that has the following properties.
\begin{enumerate}
    \item $\Phi$ is bijective.
    
    \item Every vertex of $T_1$ except  $v^0$ is an open end of exactly one auxiliary subtree used in the construction. The same is true for every vertex in $T'_1$ except $v'^0$. Whenever we refer to an auxiliary subtree below, we mean one of these, viewed  as a subtree of $T_1$ or $T'_1$.

    \item Let $v \in V(T_1) \setminus \{ v^0 \}$, let $S$ be the auxiliary subtree that has $v$ as an open end, and let $v_1, \dots, v_k$ be the  descendants of $v$  that are also open ends of $S$. Under the map $\Phi$, each descendant of $v$ is sent to either one of $\Phi(v_1),\dots \Phi(v_k)$ or  one of the descendants of  $\Phi(v), \Phi(v_1),\dots, \Phi(v_k)$.
    
    \item Let $v,S$ be as in (3), and let $w$ be a descendant of $v$ such that the geodesic from $v$ to $w$ contains an unoccupied child of $v$ with respect to $S$.  The geodesic from $\Phi(v)$ to $\Phi(w)$ passes through an unoccupied child of $\Phi(v)$ with respect to the unique auxiliary subtree $S'$ that has $\Phi(v)$ as an open end. In particular, the geodesic from $v'^0$ to $\Phi(w)$ passes through $\Phi(v)$ and leaves $S'$ at $\Phi(v)$.
    
\end{enumerate}

\subsubsection*{Step 2: Extending the map $V(T_i) \rightarrow V(T'_i)$ to $V(T_{i+1}) \rightarrow V(T'_{i+1})$.}
Suppose we have a map $\Phi \colon V(T_i) \rightarrow V(T'_i)$ which satisfies properties (1)-(4) listed above. In particular, the map $\Phi$ comes with a family of auxiliary subtrees used in its construction. Since by assumption  $T_i$ and $T'_i$ are strongly entwined in $T_{i+1}$ and $T'_{i+1}$, respectively,  the degree in $T_{i+1}$  of every vertex $v \in V(T_i)$  is strictly greater than its degree in $T_i$, and similarly for every vertex in $V(T_i')$. If $v \in V(T_i)$ and $w \in V(T_{i+1}) \setminus V(T_i)$ are connected by an edge in $T_{i+1}$, then $w$ has to be a child of $v$ (with respect to the root $v^0$), as  $v$ is  a child of a vertex in $T_i$ and so cannot be a child of $w$.

For every $v \in V(T_i)$, let $T_v^{ext}$  be the tree that contains $v$ and all vertices $w$ in $T_{i+1}$ such that the geodesic from $v$ to $w$ intersects $T_i$ exactly in the point $v$; that is, $T_v^{ext}$ contains $v$ and all its descendants that can be reached without passing through another vertex in $T_i$. Define $T_{v'}^{ext}$ analogously for all $v' \in V(T'_i)$. The map $\Phi$ induces a bijection from $\{ T_v^{ext} \mid v \in V(T_i) \}$ to $\{ T_{v'}^{ext} \mid v' \in V(T'_i) \}$ that maps $T_v^{ext}$ to $T_{\Phi(v)}^{ext}$. Rooting $T_v^{ext}$ at $v$ and $T_{v'}^{ext}$ at $v'$, we can use the construction from Step 1 to obtain  a bijective map $V( T_v ) \rightarrow V(T_{\Phi(v)})$ that extends $\Phi$. Doing this for  every $v \in V(T_i)$ yields an extension of $\Phi$ to $V(T_{i+1}) \rightarrow V(T'_{i+1})$, which we continue to denote by $\Phi$. By construction, this extension still satisfies properties (1)-(4) described in Step~1.

Since the families $T_1 \subset T_2 \subset \dots$ and $T'_1 \subset T'_2 \subset \dots$ are  filling,  every vertex in $T_{\infty}$ is contained in $T_i$ and $T'_i$ for $i$ sufficiently large. We thus obtain a map $\Phi \colon V(T_{\infty}) \rightarrow V(T_{\infty})$ which still has the properties (1)-(4) and restricts to a bijection $V(T_i) \rightarrow V(T'_i)$ for every $i$.

\subsubsection*{Step 3: The induced map on boundaries.}

The construction of $\Phi$ in Steps 1 and 2 provides two families of auxiliary subtrees $\mathcal{S}$, $\mathcal{S}'$ such that every vertex in $T_{\infty}$ is an open end of an element $S \in \mathcal{S}$ and an open end of an element $S' \in \mathcal{S}'$ except for $v^0$ and $v'^0$. The construction also provides a bijection $\mathcal{S} \rightarrow \mathcal{S}'$ that associates an auxiliary subtree $S$ of $T_i$ with the auxiliary subtree $S'$ of $T'_i$ that it was paired with when extending $\Phi$ to $S$. Abusing notation, we denote this map by $\Phi$ as well and obtain that for every $S \in \mathcal{S}$, the map $\Phi\colon V(T_{\infty}) \rightarrow V(T_{\infty})$ restricts to a bijection between the vertices of $S$ and the vertices of $\Phi(S)$.

The map $\Phi$ has the following useful property. Let $v$ be a vertex in $T_{\infty}\setminus\{v^0\}$, and let $S \in \mathcal{S}$ be the auxiliary subtree containing $v$ as an open end. Let $w$ be a descendant of $v$ such that the geodesic from $v$ to $w$ does not contain any vertex in $S$ apart from $v$. Then $\Phi(w)$ is a descendant of $\Phi(v)$ and the geodesic between them does not contain any vertex of $\Phi(S)$ apart from $\Phi(v)$.

Let $\gamma$ be a geodesic ray in $T_{\infty}$ starting at $v^0$. This geodesic ray induces a sequence of vertices $v^0 = v_0, v_1, v_2, \dots$ which it crosses in this order. We start by showing that the sequence $(\Phi(v_i))_i$ has the property that for every $R > 0$, there exists an $N$ such that for all $i \geq N$, the initial segment of length $R$ of the geodesic from $v'^0$ to $\Phi(v_i)$ is identical. Indeed, suppose without loss of generality that $R \in \mathbb{N}$, and consider the first $R$ auxiliary subtrees $S_1, \dots S_R$ that $\gamma$ leaves through one of their open ends. Let $v_{i_j}$ denote the vertex through which $\gamma$ leaves $S_j$. For all $i \geq i_j$, the geodesic from $v^0$ to $v_i$ leaves $S_j$ through the open end $v_{i_j}$ of $S_j$. By property (4) above, this implies that for all $i \geq i_j$, the geodesic from $v'^0$ to $\Phi(v_i)$ crosses $\Phi(v_{i_j})$ and leaves the auxiliary subtree that contains $\Phi(v_{i_j})$ as an open end at $\Phi(v_{i_j})$. We conclude that for all $i \geq i_R$, the geodesic from $v'^0$ to $\Phi(v_i)$ has to pass through $\Phi(v_{i_j})$ for all $1 \leq j \leq R$ and leaves the auxiliary subtree containing $\Phi(v_{i_j})$ as an open end at $\Phi(v_{i_j})$. In particular, the geodesic from $v'^0$ to $\Phi(v_{i_R})$ has length at least $R$ and for every $i \geq i_R$, the initial segment of length $R$ of the geodesic from $v'^0$ to $\Phi(v_i)$ is fixed. Therefore the sequence of geodesic segments from $v'^0$ to $\Phi(v_i)$ converges on compact sets to a geodesic ray $\gamma'$ in $T_{\infty}$. Equivalently, the sequence $(\Phi(v_i))_i$ defines a point in the visual boundary of $T_{\infty}$. Define $\partial \Phi( [\gamma] ) := [\gamma']$.

We claim that $\partial \Phi$ is a homeomorphism that restricts to a homeomorphism from $\bigcup_{i = 1}^{\infty} \partial_{\infty} T_i$ to $\bigcup_{i=1}^{\infty} \partial_{\infty}T'_i$. We start by proving bijectivity. For this we observe that $\Phi^{-1} \colon V(T_{\infty}) \rightarrow V(T_{\infty})$ is a map that can be constructed by applying Steps 1 and 2 above to the families $T_1 \subset T_2 \dots$ and $T'_1 \subset T'_2 \dots$, where the roles of $T_i$ and $T'_i$ have been switched and the same auxiliary subtrees are chosen as in the construction of $\Phi$ at every step. This implies that $\Phi^{-1}$ also satisfies properties (1)-(4) and induces a map on the boundary. Since $\partial \Phi$ is defined by sending the sequence $(v_i)_i$ that represents a point in $\partial_{\infty}T_{\infty}$ to the point represented by $(\Phi(v_i))_i$ it is straightforward to check that $\partial \Phi^{-1} = (\partial \Phi)^{-1}$.

Next, we prove that $\partial \Phi$ is a homeomorphism. By definition, the visual topology on $\partial_{\infty} T_{\infty}$ is generated by the sets $\{U_v \mid v \in V(T_{\infty}) \}$, where $U_v := \{ [\gamma] \mid \gamma(0) = v^0, \gamma \text{ crosses } v \}$. The visual topology on $\partial_{\infty} T_{\infty}$ is  generated by the sets $U'_{v'} := \{ [\gamma'] \mid \gamma'(0) = v'^0, \gamma' \text{ crosses } v'\}$.

Let $v \in V(T_{\infty})$, and let $S$ be the auxiliary subtree in $\mathcal{S}$ that contains $v$ as an open end. Define the set $D_v$ to be the set of open ends of $S$ that are descendants of $v$, including $v$ itself. Furthermore,  define $A_v \subset \mathcal{S}$ to be the set of auxiliary subtrees that contain one of the vertices in $D_v$, except for $S$ itself. Finally,  define $C_v$ to be the set of all open ends of auxiliary subtrees in $A_v$. We claim that $\partial \Phi( U_v) = \bigcup_{w \in C_v} U'_{\Phi(w)}$. Indeed, let $\xi \in U_v$, and let $\gamma$ be the geodesic ray starting at $v^0$ representing $\xi$. Since $\gamma$ crosses $v$, it leaves $S$ at a vertex contained in $D_v$. When it does so, it enters one of the auxiliary subtrees in $A_v$ and leaves this auxiliary subtree through a vertex that is contained in $C_v$. In other words, $\gamma$ crosses (exactly) one of the vertices in $C_v$. Let $(v_i)_i$ be the sequence of vertices crossed by $\gamma$. By property (3) and the construction of the boundary map discussed above, there exists an element $w \in C_v$ such that all but finitely many elements of the sequence $(\Phi(v_i))_i$ are descendants of $\Phi(w)$. In particular, $\partial \Phi(\xi) \in U'_{\Phi(w)}$.

Now suppose that $\eta \in U'_{\Phi(w)}$ for some $w \in C_v$. Let $\gamma'$ be the geodesic ray starting at $v'^0$ representing $\eta$. Since $\eta \in U'_{\Phi(w)}$, we know that $\gamma'$ crosses $\Phi(w)$. In particular, there exists some $S_o \in A_v$ that contains $w$ as an open end, and therefore, $\Phi(S_o)$ contains $\Phi(w)$ as an open end. We conclude that $\gamma'$ enters the auxiliary subtree $\Phi(S_o)$ and has to leave it through one of its open ends. Since the open ends of $\Phi(S_o)$ are the images of the open ends of $S_o$ under $\Phi$, there exists some $\tilde{w} \in C_v$ such that $\gamma'$ leaves $\Phi(S_o)$ through $\Phi(\tilde{w})$. In particular, since $\Phi^{-1}$ satisfies property (3), every vertex crossed by $\gamma'$ after crossinge $\Phi(\tilde{w})$ is the image of a descendant of $\tilde{w}$. We conclude that $\Phi^{-1}(\eta) \in U_{\tilde{w}} \subset U_v$. Therefore, $\partial \Phi(U_v) = \bigcup_{w \in C_v} U'_{\Phi(w)}$.

This implies that $\partial \Phi$ sends open sets to open sets. Since $(\partial \Phi)^{-1}$ is the boundary map of $\Phi^{-1}$, which has the same properties as $\Phi$, we conclude that $(\partial \Phi)^{-1}$ sends open sets to open sets as well. We conclude that $\partial \Phi$ is a homeomorphism.

Finally, we check what $\partial \Phi$ does to the subspace $\bigcup_{i=1}^{\infty} \partial_{\infty} T_i$. Since the sequences $(T_i)_i$ and $(T'_i)_i$ are nested and the restriction of $\Phi$ to $T_i$ is a bijection $V(T_i) \rightarrow V(T'_i)$, we conclude that $\Phi$ sends $\partial_{\infty} T_i$ into $\partial_{\infty} T'_i$. Since the restriction of $\Phi$ to any $T_i$ is bijective, $\Phi^{-1}$ sends $V(T_i')$ to $V(T_i)$, and therefore, $\partial \Phi$ sends $\partial_{\infty} T_i$ onto $\partial_{\infty} T'_i$. We conclude that the restriction of $\partial \Phi$ induces a bijection $\bigcup_{i=1}^{\infty} \partial_{\infty} T_i \rightarrow \bigcup_{i=1}^{\infty} \partial_{\infty} T'_i$. This completes the proof.

\subsection{Comparison with the visual topology} \label{subsec:Comparisonwithvisualtopology}
Since $X$ is a non-positively curved manifold, its $\kappa$-Morse boundary injects into the visual boundary. We can thus compare the visual topology on $\partial_{\kappa} X$ with the topology $\mathcal{T}(G, [T_{\infty}])$. It is easy to see that convergence in visual topology implies convergence in $\mathcal{T}(G, [T_{\infty}])$, implying that $\mathcal{T}(G, [T_{\infty}]) \subseteq \Vis$. We conjecture that the converse is true as well.

\begin{conjecture} \label{conjecture:Visualtopologyongraphmanifolds}
Let $G$ be the fundamental group of a non-positively curved graph manifold and $\mathcal{T}(G, X, T_{\infty})$ the topology induced by the action of $G$ on its Bass-Serre tree. Then $\mathcal{T}(G, X, T_{\infty})$ coincides with the visual topology on $\partial_{\kappa} X$.
\end{conjecture}

We make a few comments on the potential proof of this conjecture. There are two main approaches, which are based on two different auxiliary topologies that can be defined on $\partial_{\kappa} X$. In \cite{PetytSprianoZalloum22}, Petyt, Spriano, and Zalloum introduce the notion of a \textit{curtain} that allows one to define a \textit{curtain topology} on the Morse boundary $\partial_{1} X$ that coincides with the visual topology. (Petyt, Spriano, and Zalloum have informed us that there is an ongoing work to generalise these results to $\kappa$-Morse boundaries.) One is thus left to show that the curtain topology coincides with $\mathcal{T}(G, X, T_{\infty})$ on the $\kappa$-Morse boundary. Proving this relies on a careful understanding how curtains intersect different pieces of $X$, and specifically if certain curtains that are separated by sufficiently many pieces are $L$-separated in the sense introduced in \cite{PetytSprianoZalloum22}.

The other approach replaces the curtain topology with the median topology introduced in Section \ref{sec:Topologyforgroupswithacoarsemedianstructure}. Since the fundamental groups of graph manifolds are hierarchically hyperbolic groups (as shown in \cite{HagenRussellSistoSpriano22}), we can equip $X$ with the coarse median structure induced by the hierarchical structure. Since the hyperbolic projection $(G, X, T_{\infty})$ is $\kappa$-injective, it is also $(\kappa, \mu)$-injective for this median, and the induced inclusion and topology are the same. (Note that this is true because for hierarchically hyperbolic groups, $\partial_{\kappa} X$ has $\mu$-median representatives, as discussed at the beginning of section \ref{subsec:Definingthetopologyinthemediancase}.) As we have shown in this paper, the median topology coincides with $\mathcal{T}(G, X, T_{\infty})$, and so one is left to show that the median topology coincides with the visual topology on $\partial_{\kappa} X$. Similarly to the treatment via curtains, one has to be careful when dealing with the existence of flat sectors in $X$ that cross infinitely many pieces.

Either of these arguments is likely to generalise to all CK-admissible groups, which is why we conjecture that for all geometric actions of CK-admissible groups on a Hadamard space X (geodesically complete $\CAT$ spaces), the $\kappa$-Morse boundary of $X$ with the visual topology is homeomorphic to a Cantor chain.




\appendix

\section{Unparametrised quasi-rulers and hyperbolic spaces} \label{sec:unparametrisedquasigeodesicsandhyperbolicspaces}

Throughout this paper, we frequently consider hyperbolic spaces that are not necessarily geodesic. One complication with such spaces is that hyperbolicity is not necessarily preserved under quasi-isometries (cf. \cite{BS}). Similarly, it is also not clear whether quasi-isometries extend to homeomorphisms between boundaries. The purpose of this appendix is to prove the following result, which deals with this problem for a nice class of hyperbolic spaces that includes all the spaces we deal with in this paper.

\begin{corollary} \label{cor:quasiisometriesofquasiruledspacesextend}
Let $f \colon X \rightarrow X'$ be a quasi-isometry between metric spaces that admit a $D$-quasi-ruling, and suppose $X$ is hyperbolic. Then $X'$ is hyperbolic, and $f$ naturally extends  to a homeomorphism $\partial f \colon \partial_{\infty} X \rightarrow \partial_{\infty} X'$. Furthermore, for every path $\gamma$ in $X$ defining a point in $\partial_{\infty} X$, the composition $f \circ \gamma$ defines a point in $\partial_{\infty} X'$ and $[f \circ \gamma] = \partial f([\gamma])$.
\end{corollary}

The notion of $D$-quasi-rulings, which we make precise below, is inspired by work of Blach\`ere, Ha\"issinsky, and Mathieu in \cite{BlachereHaissinskyMathieu11}. In fact, they proved the first statement in the corollary using their notion of `parametrised' quasi-rulers, which is dependent on a choice of parametrisation. This does not suit our purposes, as we frequently consider projections of quasi-geodesics, for which we need properties that are invariant under reparametrisation. For this reason we will modify their definition. We start by giving their original definition of parametrised quasi-rulers and discuss how it relates to the notion of unparametrised quasi-rulers that we will be using.

\begin{definition}[\cite{BlachereHaissinskyMathieu11}]
Let $X$ be a metric space, and fix $\tau \geq 0$. A \textit{parametrised $\tau$-quasi-ruler} is a quasi-geodesic $\gamma \colon I \rightarrow X$ such that for all $t  < s < r$, we have
\[ d( \gamma(t), \gamma(s) ) + d( \gamma(s), \gamma(r) ) \leq d( \gamma(t), \gamma(r)) + \tau. \]
Given constants $K \geq 1$, $C, \tau \geq 0$, a {\it quasi-ruling structure} on $X$ is a set $\mathcal{G}$ of $(K, C)$-quasi-geodesics that are parametrised $\tau$-quasi-rulers such that any two points in $X$ can be connected by an element of $\mathcal{G}$. 
We say that $X$ is {\it $(K,C)$-parametrised $\tau$-quasi-ruled} if 
for any $x, y \in X$, there exists a $(K, C)$-quasi-geodesic from $x$ to $y$ and every $(K, C)$-quasi-geodesic in $X$ is a parametrised $\tau$-quasi-ruler.
\end{definition}

The following modification of this definition  is independent of the parametrisation.

\begin{definition} \label{def:unparametrisedquasirulers}
Let $X$ be a metric space, $D \geq 0$, and $I \subset \mathbb{R}$ a (possibly unbounded) closed interval. A path $\gamma \colon I \rightarrow X$ is an {\it unparametrised $D$-quasi-ruler} if it satisfies the following conditions.
\begin{enumerate}
    \item $\forall t < s < r$, we have $d( \gamma(t), \gamma(s) ) + d( \gamma(s), \gamma(r) ) \leq d( \gamma(t), \gamma(r) ) + D$.
    
    \item $\forall t_0 \in I$, we have $\limsup_{ \vert t - t_0 \vert \rightarrow 0} d( \gamma(t), \gamma(t_0) ) < D$.
\end{enumerate}

If $I$ is a bounded interval, we say $\gamma$ is {\it finite}. If $I$ is unbounded in one direction and the image of $\gamma$ is unbounded, we say $\gamma$ is {\it unbounded}. If $I = \mathbb{R}$ and the images of $(-\infty, 0]$ and $[0, \infty)$ are unbounded, we say $\gamma$ is {\it unbounded in both directions}.

A space {\it admits a $D$-quasi-ruling} if there exists $D \geq 0$ such that there is an unparametrised $D$-quasi-ruler between any two points.
\end{definition}

We sometimes refer to property (2) in Defintion~\ref{def:unparametrisedquasirulers} by saying that $\gamma$ does not make any jumps of size $D$ or larger. Observe that any reparametrisation of an unparametrized $D$-quasi-ruler is again an unparametrized $D$-quasi-ruler, as any reparametrisation preserves the order of all $t < s < r$. (Here, a reparametrisation is an orientation-preserving homeomorphism $\varphi \colon I \rightarrow I$ on the domain of the quasi-ruler.)

The connection between $\tau$-quasi-rulers and unparametrised $D$-quasi-rulers is given by the following result.

\begin{lemma} \label{lem:UniformQuasigeodesicParametrisation}
Let $D \geq 0$. Then there exist constants $K, C \geq 0$ depending only on $D$ such that every unparametrised $D$-quasi-ruler that is either finite, unbounded, or unbounded in both directions can be parametrised as a $(K, C)$-quasi-geodesic.
\end{lemma}

\begin{remark}
The requirement that the unparametrised $D$-quasi-ruler  be finite, unbounded, or unbounded in both directions is necessary. The reason is that any bounded path $\gamma \colon [0, \infty) \rightarrow X$ is an unparametrised $D$-quasi-ruler for sufficiently large $D$. If such a path could be reparametrised as a quasi-geodesic ray, that ray would have to be unbounded, which is a contradiction.
\end{remark}

\begin{proof}
Let $\gamma$ be an unparametrised $D$-quasi-ruler and fix $0 < \epsilon < D$. Observe that for all points $x, y, z$ on $\gamma$, such that $y$ lies between $x$ and $z$, we have
\[ d( x,y) \leq d(x,y) + d(y,z) \leq d(x,z) + D. \]

We first deal with the case where $\gamma$ is unbounded. We will inductively construct a sequence of points $x_n$ on $\gamma$ that will allow us to construct a suitable parametrisation. Let $x_1$  be the start of $\gamma$. Suppose  $x_1, \dots, x_j$ have been chosen, and suppose that for all $1 \leq i \leq j-1$, we have $D + \epsilon \leq d(x_i, x_{i+1}) \leq 2D + \epsilon$. Since $\gamma$ is unbounded by assumption, there exists $x \in \gamma$ after $x_j$ such that $d(x_j, x) > D + \epsilon$. Let $x_{inf}$ be the infimum of all such points. (Since $\gamma$ is a path, we can choose any parametrisation and take the infimum of all $t$ such that $\gamma(t)$ is after $x_j$ and has $d(x_j, \gamma(t)) > D + \epsilon$. The resulting point on $\gamma$ is $x_{inf}$ and does not depend on the choice of parametrisation.) If $d(x_j, x_{inf}) \leq D + \epsilon$, choose $x$ on $\gamma$ slightly after $x_{inf}$ such that $d(x_j, x) > D + \epsilon$ and $d(x_{inf}, x) \leq D$; such $x$ exists as $\gamma$ makes no jumps larger than $D$. Then $d(x_j, x) \leq 2D + \epsilon$, and we set $x_{j+1} := x$. On the other hand, if $d(x_j, x_{inf}) > D + \epsilon$, then $d(x_j, x_{inf}) \leq 2D + \epsilon$, as $\gamma$ makes no jumps larger than $D$, and we set $x_{j+1} := x_{inf}$. In both cases, we obtain a point $x_{j+1}$ such that $D+ \epsilon \leq d( x_j, x_{j+1}) \leq 2D+\epsilon$, which recovers the induction assumption.
 
Observe that for every $i$, each point  $x$ on  $\gamma$ between $x_i$ and $x_{i+1}$ is in the $(3D+\epsilon)$-neighbourhood of the set $\{ x_i, x_{i+1} \}$, because
\[ d(x_i, x) + d(x, x_{i+1}) \leq d(x_i, x_{i+1} ) + D \leq 3D + \epsilon. \]
We conclude that the segment of $\gamma$ between $x_1$ and $x_j$ is contained in the $(3D+\epsilon)$-neighbourhood of the set $\{ x_1, \dots, x_j \}$.

This inductive procedure yields a sequence $(x_j)_j$ of points on $\gamma$ that is consecutive with respect to an orientation of $\gamma$. In addition,   $D + \epsilon < d(x_j, x_{j+1}) \leq 2D + \epsilon$ for all $j \geq 1$, and 
\[ d(x_i, x_j) + d(x_j, x_k) \leq d(x_i, x_k) + D, \]
for all $i \leq j \leq k$, 
as $\gamma$ is an unparametrised $D$-quasi-ruler. Therefore, for all $i, k \in \mathbb{N}^{+}$, we have
\[ d(x_i, x_k) \leq \sum_{j = i}^{k-1} d(x_j, x_{j+1}) \leq \vert i - k \vert (2D + \epsilon), \] and
\[ d(x_i, x_k) \geq \sum_{j = i}^{k-1} ( d(x_j, x_{j+1}) ) - ( \vert k - i \vert - 1 ) D \geq \vert k-i \vert (D + \epsilon - D) = \vert k-i \vert \epsilon. \]

We would like to reparametrise $\gamma$ such that we reach the point $x_i$ at time $i$. In order for this to provide the basis of a reparametrisation of $\gamma$, we need to make sure that for every $x \in \gamma$, there exists some $x_i$ such that $\gamma$ reaches $x$ before it reaches $x_i$. Suppose this was not the case. Then $x_i$ lies between $x_1$ and $x$ for all $i \geq 1$. In particular, we can estimate
\[ d(x_1, x) \geq d(x_1, x_i) + d(x_i, x) - D \geq \vert i - 1 \vert \epsilon + d(x_i, x) - D \xrightarrow{i \rightarrow \infty} \infty. \]
This is a contradiction, and we conclude that for every $x \in \gamma$, only finitely many $x_i$ are crossed by $\gamma$ before crossing $x$. In particular, $x$ is contained in the $(3D + \epsilon)$-neighbourhood of $\{ x_1, \dots, x_j \}$ for all but finitely many $j$. This implies that $\gamma$ is contained in the $(3D + \epsilon)$-neighbourhood of $\{ x_1, \dots \}$.

We now define $\gamma'$ to be a reparametrisation of $\gamma$ such that $\gamma'(i) = x_i$ for all $i \in \mathbb{N}^{+}$ and such that $\gamma' \vert_{[i, i+1]}$ is any reparametrisation of the segment of $\gamma$ between $x_i$ and $x_{i+1}$. Since for every point $x \in \gamma$ there exists some $x_j$ which appears after $x$, this defines a reparametrisation of $\gamma$ (rather than a reparametrisation of a segment of $\gamma$). Due to the properties of the $x_i$ established above, we see that $\gamma'$ is a $(\max( 2D + \epsilon, \frac{1}{\epsilon} ), 2(3D + \epsilon))$-quasi-geodesic. These parameters only depend on $D$ and a choice of $\epsilon$, as desired.

If $\gamma$ is bi-infinite, choose $x_1$ to be any point on $\gamma$ and perform the induction above in both directions. This yields a sequence $(x_i)_{i \in \mathbb{Z}}$ of consecutive points on $\gamma$ with the same properties as above. One then analogously obtains a quasi-geodesic parametrisation, which proves the lemma for all unparametrised $D$-quasi-rulers that are unbounded or unbounded in both directions.

Finally, suppose $\gamma$ is finite. As shown earlier, we have for every $t \in [a,b]$ that $d(\gamma(a), \gamma(t)) \leq d(\gamma(a), \gamma(b)) + D$. We conclude that the image of $\gamma$ has finite diameter.

We now use the same inductive procedure as above to obtain a sequence $x_1, \dots x_j$. We claim that, after finitely many steps, the entire path $\gamma$ is contained in the $(3D + \epsilon)$-neighbourhood of $\{ x_1, \dots x_j \}$. To show this, consider the set
\begin{equation*}
\begin{split}
    \mathcal{T} := \{ t \in [a, b] \mid \,& \exists x_1, \dots x_j \in \Ima(\gamma) : \forall 1 \leq i \leq j-2,\\
    & D + \epsilon \leq d(x_i, x_{i+1}) \leq 2D+\epsilon, d( x_{j-1}, x_j) \leq 2D+\epsilon\\
    & \qquad \qquad \qquad \qquad \qquad \quad \text{ and } \gamma(t) \text{ lies between } x_1, x_j \}.
\end{split}
\end{equation*}
We highlight that, when we say that $\gamma(t)$ lies between $x_1, x_j$, we also allow $\gamma(t) \in \{ x_1, x_j \}$. We claim that $\mc T$ is the entire interval $[a, b]$. The set $\mc T$ is open, as $\gamma$ does not make jumps of size $D$ or larger, which allows us to always extend any sequence $x_1, \dots, x_j$ by another point (unless $x_j = \gamma(b)$, in which case we would be done). Let $T$ be the supremum of this set. Since $\gamma$ does not make jumps of size $D$ or larger, there exists $T_{-}$ between $\gamma(a)$ and $\gamma(T)$ such that $d( \gamma(T_{-}), \gamma(T)) \leq D$. Since $T$ is the supremum, we find a sequence $x_1, \dots, x_j$ with $x_j$ between $\gamma(T_{-})$ and $\gamma(T)$ such that for all $1 \leq i \leq j-2$, $D + \epsilon \leq ( x_i, x_{i+1}) \leq 2D+ \epsilon$ and $d(x_{j-1}, x_j) \leq 2D+\epsilon$. Since $x_j$ is between $\gamma(T_{-})$ and $\gamma(T)$, we also have $d(x_j, \gamma(T)) \leq 2D+ \epsilon$ and therefore, we either add $x_{j+1} := \gamma(T)$ to the sequence or put $\tilde{x_j} := \gamma(T)$, depending on whether $d(x_{j-1}, x_j) \geq D+\epsilon$ or not. In both cases, we obtain that $T \in \mathcal{T}$. We conclude that $\mathcal{T}$ is open and closed and therefore it contains $b$.

This implies that there exists a sequence $x_1, \dots, x_j$ such that $\gamma(a) = x_1$, $\gamma(b) = x_j$, for all $1 \leq i \leq j-2$, $D + \epsilon \leq d(x_i, x_{i+1}) \leq 2D+\epsilon$ and $d( x_{j-1}, x_j) \leq 2D+\epsilon$. In particular, the entire path $\gamma$ is contained in the $(3D+\epsilon)$-neighbourhood of the set $\{ x_1, \dots, x_j \}$.

The same estimates as before show that for all $i, k \in \{1, \dots, j-1 \}$, we have
\[ d(x_i, x_k) \leq \vert i - k \vert (2D + \epsilon) \]
and
\[ d(x_i, x_k) \geq \vert i - k \vert \epsilon. \]
Again, we define $\gamma'$ to be a reparametrisation of $\gamma$ such that $\gamma'(i) = x_i$ for all $i \in \{1, \dots, j\}$ and choose any parametrisation on the intervals in between. The estimates above show that the restriction of $\gamma'$ to $[1, j-1]$ is a $(\max(2D + \epsilon, \frac{1}{\epsilon}), 2(3D+\epsilon))$-quasi-geodesic. Since the entire segment of $\gamma$ between $x_{j-1}$ and $x_{j}$ is contained in a ball of radius $3D+\epsilon$ around $x_{j-1}$, we conclude that $\gamma'$ is a $(\max(2D + \epsilon, \frac{1}{\epsilon}), C)$-quasi-geodesic for $C=2(3D+\epsilon) + 3D + \epsilon + \frac{1}{ \max(2D+\epsilon, \frac{1}{\epsilon} ) }$. This proves the lemma.
\end{proof}

Lemma~\ref{lem:UniformQuasigeodesicParametrisation} shows that a metric space $X$ admits a $D$-quasi-ruling if and only if it admits a quasi-ruling structure with parameters $D, K, C$
, where $K, C$ depend only on $D$. Since all unparametrised $D$-quasi-rulers that we will consider admit a parametrisation as a parametrised $D$-quasi-ruler by the previous lemma, we drop the `unparametrised' from now on.

We now prepare to prove Corollary \ref{cor:quasiisometriesofquasiruledspacesextend}. The basic idea behind the proof is that every hyperbolic space that admits a $D$-quasi-ruling is $(1, C)$-quasi-isometric to a geodesic metric space. Since $(1, C)$-quasi-isometries always preserve hyperbolicity and extend to homeomorphisms between boundaries, this reduces the corollary to known results about geodesic hyperbolic spaces. The key step is the following lemma.

\begin{lemma} \label{lem:Makingquasiruledspacesgeodesic}
If $X$ is a metric space that admits a $D$-quasi-ruling structure, then there exists a geodesic space $X_g$ such that $X$ is $(1, 4D)$-quasi-isometric to $X_g$.
\end{lemma}

\begin{proof}
Let $(X, d)$ be a quasi-ruled metric space. We start by constructing a geodesic metric space $(X_g, d')$ as follows. The space $X_g$ will be a graph.  We begin by considering the set vertices of $X$; these will all be vertices of $X_g$.   We then define a collection of ``lines'' connecting vertices of $X_g$.  First, we add a line between any two points $x, y \in X$, denoted $\overline{xy}$, of length $\max \left\{ \frac{D}{2}, d(x,y) \right\}$, which we think of as a path from $x$ to $y$ parametrised by arc-length. The notation $\overline{yx}$ denotes the same line with reverse orientation.

Next, we add  vertices (that do not correspond to points of $X$) to subdivide each line $\overline{xy}$ into edges and add additional edges as follows.  Choose an ordering $(x,y)$ for any pair of points $x,y\in X$, and let $\gamma \colon I \rightarrow X$ be a $D$-quasi-ruler from $x$ to $y$. Since any $D$-quasi-ruler makes jumps of size at most $D$, we obtain that for every integer $0 < k < \frac{d(x,y)}{D} - \frac{1}{2}$, there exists a point $x_k$ in the image of $\gamma$ such that $d(x, x_k) \in [ kD - \frac{D}{2}, kD + \frac{D}{2}]$. We now add a line of length $2D$ from $x_k$ to $\overline{xy}(kD)$. The resulting space is a graph $X_g$ whose edges all have length at least $\frac{D}{2}$. Note that the line $\overline{xy}$ is cut into edges of length $D$, except for the edge closest to $y$, whose length lies between $\frac{D}{2}$ and $\frac{3D}{2}$. The choice of ordering of $x$ and $y$ determines on which end of the line we have an edge whose length isn't known precisely. We denote the induced path-metric on $X_g$ by $d'$.

We claim that the pair $(X_g, d')$ is a geodesic metric space that is $(1, 4D)$-quasi-isometric to $(X, d)$. We first show that $(X_g, d')$ is geodesic. We start by showing that for every $x, y \in X$, the line $\overline{xy}$ in $X$ is a geodesic from $x$ to $y$. Suppose $\gamma$ is a path in $X_g$ from $x$ to $y$. By definition, $d'(x, y)$ is the infimum of the length of any such path. Since $X_g$ is a graph, any path between vertices can be shortened to be an edge path, i.e., $\gamma = e_1 * e_2 * \dots * e_n$ is a finite concatenation of edges in $X_g$. Every edge in $X_g$ belongs to one of the following three types:
\begin{enumerate}
    \item[(1)] An edge between points $x, y \in X$ that satisfy $d(x, y) \leq D$. These edges have length between $\frac{D}{2}$ and $D$.
    
    \item[(2)] A segment of a line between points $x, y \in X$ that satisfy $d(x, y) > D$. These edges either have length $D$ or length between $\frac{D}{2}$ and $\frac{3D}{2}$.
    
    \item[(3)] An edge that connects a point $x' \in X$ with a point on a line $\overline{xy}$ for some $x, y \in X$ with $d(x, y) > D$. These edges have length $2D$.
\end{enumerate}

We claim that if $\gamma$ contains an edge of type $(3)$, then $\gamma$ can be shortened. Suppose $e_i$ is an edge of type $(3)$. Reversing the orientation of $\gamma$ if necessary, we may assume without loss of generality that $e_i$ starts at a point $x' \in X$ and ends on a vertex $\overline{xy}(kD)$ for some $x, y \in X$ and $0 < k < \frac{d(x,y)}{D} - \frac{1}{2}$. There are three cases.



Case 1: The edges $e_{i+1} * \dots * e_{i+l}$ form a segment in the interior of $\overline{xy}$ and the edge $e_{i+l+1}$ is an edge from a vertex $\overline{xy}( (k \pm l)D )$ to some $y' \in X$. Without loss of generality, assume $y'$ is connected to $\overline{xy}( (k+l) D)$. In this case,
$ l(e_i * \dots * e_{i+l+1} ) = 4D + lD,$
and since $x', y'$ lie on a $D$-quasi-ruler from $x$ to $y$, we estimate
\begin{equation*}
    \begin{split}
        d(x', y') & \leq d(x, y') - d(x, x') + D\\
        & \leq (k+l) D + \frac{D}{2} - kD + \frac{D}{2} + D\\
        & \leq lD + 2D < l(e_i * \dots * e_{i+l+1}).
    \end{split}
\end{equation*}
We conclude that we can shorten $\gamma$ by replacing $e_i * \dots * e_{i+l+1}$ by the line $\overline{x'y'}$.

Case 2: The edges $e_{i+1} * \dots * e_{i+l}$ form the segment in $\overline{xy}$ from $\overline{xy}(kD)$ to $y$. This case is similar and easier.

Case 3: The edges $e_{i+1} * \dots * e_{i+l}$ form the segment in $\overline{xy}$ from $\overline{xy}(kD)$ to $x$. This case is analogous to the second case.

We conclude that, given a path $\gamma$ from $x$ to $y$, we can shorten this path to be an edge path which does not use any edges of type $(3)$. In other words, $\gamma$ can be written as a concatenation of lines $\overline{x_i x_{i+1}}$ for points $x = x_0, x_1, \dots, x_n = y$ in $X$. By the triangle inequality for the metric $d$, we also know that
\begin{equation*}
    \begin{split}
        l(\gamma) & = \sum_{i=1}^n l( \overline{x_{i-1} x_i} )
         = \sum_{i=1}^n \max\left\{ d(x_{i-1}, x_i), \frac{D}{2} \right\}
         \geq \sum_{i=1}^n d(x_{i-1}, x_i)
         \geq d(x, y).
    \end{split}
\end{equation*}
This implies two things. First, for any path from $x$ to $y$, its length is at least $d(x,y)$ and therefore $d'(x,y) \geq d(x,y)$. Second, whenever $d(x,y) \geq \frac{D}{2}$, the line $\overline{xy}$ is a geodesic, as its length is equal to $d(x,y)$. Now suppose $d(x,y) < \frac{D}{2}$ and let $\gamma$ be a path from $x$ to $y$ in $X_g$. After shortening $\gamma$ to be an edge path, we observe that every edge in $\gamma$ contributes at least $\frac{D}{2}$ to the length of $\gamma$. Thus, the shortest path is given by the line $\overline{xy}$, which has exactly length $\frac{D}{2}$. We conclude that for any two points $x, y \in X$, the line $\overline{xy}$ is a geodesic in $X_g$ from $x$ to $y$ with respect to $d'$. Therefore, any two points in $X$ can be connected by a geodesic in $X'$.

To finish the proof that $(X_g, d')$ is geodesic, consider now two points $p, q \in X_g$, which may not lie in $X$. Consider the maximal subtree $\Gamma_p$ of $X_g$ that contains $p$ and such that the only vertices in $\Gamma_p$ that lie in $X$ are leaves of $\Gamma_p$. It is straightforward to see that $\Gamma_p$ is the line $\overline{xy}$ for some $x, y \in X$ together with the edges of type (3) that are attached to $\overline{xy}$. This subtree is finite and we have an analogous finite subtree containing $q$. Consider now leaves $v$ of $\Gamma_p$ and $w$ of $\Gamma_q$. There are finitely many such pairings and each pairing induces a unique path from $p$ to $q$, which is the shortest path from $p$ to $q$ that passes through $v$ and $w$. The shortest of these paths is length-minimizing for $p$ and $q$. Therefore, there is a path that realizes the infimum of all lengths of paths from $p$ to $q$ and this path is a geodesic with respect to $d'$. We conclude that $(X_g, d')$ is a geodesic metric space.

We now show that the inclusion $X \hookrightarrow X_g$ is a $(1, 4D)$-quasi-isometry. Let $x, y \in X$. We already know that $d(x,y) \leq d'(x,y)$. Knowing the geodesics in $X'$ from above, we see that
\[ d'(x, y) = l( \overline{xy} ) = \max\{ d(x,y), \frac{D}{2} \} \leq d(x,y)+\frac{D}{2}. \]
Thus $X\hookrightarrow X_g$ is a $(1, \frac{D}{2})$-quasi-isometric embedding.

To show the inclusion is coarsely surjective, let $p \in X_g$. By construction, every edge in $X_g$ has length at most $2D$, and every vertex in $X_g$ either lies in $X$ or is connected to a point in $X$ by an edge of type (3) that has length $2D$. We conclude that every points $p \in X_g$ has distance at most $4D$ from a point in $X$. Therefore, the inclusion is coarsely surjective and a quasi-isometry.
\end{proof}

We will show that if $X$ is hyperbolic, then $X_g$ is also hyperbolic.  This relies on the more general fact that a $(1,C)$--quasi-isometry $f\colon X\to X'$ between metric spaces preserves hyperbolicity. Using the Gromov product definition of hyperbolicity (see Section~\ref{sec:Preliminaries}), this follows as for all $x, y, o \in X$,
\[ (x \mid y)_o- 3C \leq (f(x) \mid f(y))_{f(o)} \leq (x \mid y)_o + 3C. \]    Moreover, it is straightforward to show from this inequality that $f$  naturally extends to a continuous map $\partial f\colon \partial_{\infty} X \rightarrow \partial_{\infty} X'$.  The same reasoning using any quasi-inverse $f^{-1}$ of $f$ also shows that $\partial (f^{-1}) = ( \partial f )^{-1}$.  This leads to the following corollary.

\begin{corollary} \label{cor:Makingquasiruledspaceshyperbolic}
Let $X$ be a hyperbolic space that admits a quasi-ruling structure. Then there exists a $(1, C)$-quasi-isometry $\iota \colon X \rightarrow X_g$ to a hyperbolic geodesic metric space $X_g$ that extends to a homeomorphism $\partial \iota: \partial_{\infty} X \rightarrow \partial_{\infty} X_g$. 
\end{corollary}

In order to conclude the result that we will use later, we need one auxiliary definition from \cite{BS} and two of its basic properties.

\begin{definition}
Let $(X, d)$ be a metric space and $o \in Z$. For $x, y, z, u \in X$, denote
\[ \langle x, y, z, u  \rangle := (x \mid y)_o + (z \mid u)_o - (x \mid z)_o - (y \mid u)_o, \]
which is independent of $o$. A map $f \colon X \rightarrow X'$ between metric spaces is {\it strongly PQ-isometric} if there exist constants $K, C \geq 0$ such that for all $x, y, z, u \in X$ with $\langle x, y, z, u \rangle \geq 0$, we have
\[ \frac{1}{K} \langle x, y, z, u \rangle - C \leq \langle f(x), f(y), f(z), f(u) \rangle \leq K \langle x, y, z, u \rangle + C. \]
\end{definition}

\begin{lemma}[{\cite[Proposition 4.3.2]{BS}}] \label{lem:GromovProductcoarselypreserved}
A strongly PQ-isometric map $f \colon X \rightarrow X'$  is a quasi-isometric embedding, and there exist constants $K, C \geq 0$ such that for all $x, y, o \in X$,
\[ \frac{1}{K} (x \mid y)_o - C \leq (f(x) \mid f(y))_{f(o)} \leq K(x \mid y)_o. \]
\end{lemma}

By the same argument as for Corollary \ref{cor:Makingquasiruledspaceshyperbolic},  any strongly PQ-isometric map between hyperbolic spaces extends to a continuous map between their boundaries. The following is the key result that makes geodesic hyperbolic spaces so much better behaved than general hyperbolic spaces.

\begin{proposition}[{\cite[Corollary 4.4.2]{BS}}] \label{prop:quasiisometriesarePQisometries}
A map $f \colon X \rightarrow X'$ between geodesic hyperbolic spaces is a quasi-isometric embedding if and only if it is strongly PQ-isometric.
\end{proposition}

We are now ready to prove Corollary \ref{cor:quasiisometriesofquasiruledspacesextend}.



\begin{proof}[Proof of Corollary \ref{cor:quasiisometriesofquasiruledspacesextend}]
Let $\iota \colon X \rightarrow X_g$ and $\iota' \colon X' \rightarrow X'_g$ be the $(1, C)$-quasi-isometries constructed in Lemma \ref{lem:Makingquasiruledspacesgeodesic} and let $\iota^{-1}$ and $\iota'^{-1}$ denote quasi-inverses of $\iota$ and $\iota'$ respectively. Then $\iota' \circ f \circ \iota^{-1}$ is a quasi-isometry between geodesic metric spaces. Since $X$ is hyperbolic, Corollary \ref{cor:Makingquasiruledspaceshyperbolic} implies that $X_g$ is hyperbolic. Since quasi-isometries between geodesic spaces preserve hyperbolicities, we conclude that $X'_g$ is hyperbolic and thus also $X'$ is hyperbolic, as any $(1, C)$-quasi-isometry has a $(1, \tilde{C})$-quasi-inverse.

Since $X_g$ and $X'_g$ are geodesic hyperbolic spaces and $\iota' \circ f \circ \iota^{-1}$ is a quasi-isometry, this map is also PQ-isometric by Proposition \ref{prop:quasiisometriesarePQisometries}. In particular, $\iota' \circ f \circ \iota^{-1}$ satisfies the inequalities in Lemma \ref{lem:GromovProductcoarselypreserved}. Since $\iota, \iota'^{-1}$ are $(1, C)$-quasi-isometries, they satisfy the same inequalities. We obtain that the composition $\iota'^{-1} \circ \iota' \circ f \circ \iota^{-1} \circ \iota$ also satisfies the inequalities of Lemma \ref{lem:GromovProductcoarselypreserved}, albeit for different constants $K, C$. To simplify notation, we let $\tilde{f} := \iota'^{-1} \circ \iota' \circ f \circ \iota^{-1} \circ \iota$. It follows from the discussion above that there exist $K, C, D \geq 0$ such that this map satisfies the following two estimates:
\begin{equation} \label{eq:Equation2}
\forall x \in X : d( f(x), \tilde{f}(x)) \leq D,
\end{equation}
\begin{equation} \label{eq:Equation3}
\forall x, y, o \in X : \frac{1}{K} (x \mid y)_o - C \leq ( \tilde{f}(x) \mid \tilde{f}(y) )_{\tilde{f}(o)} \leq K (x \mid y)_o + C.
\end{equation}

Now let $(x_i)_i$ represent a point in the boundary of $X$, i.e., $(x_i)_i$ converges to infinity. By  (\ref{eq:Equation3}), the sequence $(\tilde{f}(x_i))_i$ converges to infinity in $X'$. Furthermore,  (\ref{eq:Equation3}) also implies that any two equivalent sequences $(x_i)_i$ and  $(x'_i)_i$ converging to infinity in $X$ are sent to equivalent sequences $(\tilde{f}(x_i))_i$ and $(\tilde{f}(x'_i))_i$ in $X'$. By  (\ref{eq:Equation2}), the same holds for the map $f$. 

Therefore $f$ induces a map on boundaries $\partial f \colon \partial_{\infty} X \rightarrow \partial_{\infty} X'$. The definition of the visual topolog and  (\ref{eq:Equation3}) together imply that $\partial f$ is continuous. Applying the same argument to a quasi-inverse of $f$ shows that $\partial f$ is invertible and its inverse is continuous as well. Therefore, $\partial f$ is a homeomorphism.

Finally, if $\gamma$ is a path in $X$ that defines a point in $\partial_{\infty}X$, then any two increasing, diverging sequences $(\gamma(t_i))_i, (\gamma(s_i))_i$ converge to infinity and are equivalent. Therefore, their images under $f$ converge to infinity and are equivalent as well. We conclude that $f \circ \gamma$ defines a point in $\partial_{\infty} X'$ and that point equals $\partial f( [(\gamma(t_i))_i ])$ for any increasing, diverging sequence on $\gamma$. This finishes the proof.
\end{proof}

We also use the following elementary fact about $D$-quasi-rulers and their Gromov products.

\begin{lemma} \label{lem:ControllingdistanceviaGromovproduct}
Let $X$ be a $\delta$-hyperbolic space with basepoint $o \in X$, and let $\gamma, \gamma' \colon [0, \infty) \rightarrow X$ be two unbounded $D$-quasi-rulers that start at $o$ and define points $\xi, \eta \in \partial_{\infty} X$. If $x' \in \gamma$ and $y' \in \gamma'$ are such that $d(o, x'), d(o, y') \leq (\xi \mid \eta)_o + D$, then
\[ ( x' \mid y')_o \geq \min\{ d(o,x'), d(o,y') \} - D - 2\delta, \] and 
\[ d(x', y') \leq D + 2 \delta + \vert d(o,x') - d(o,y') \vert. \]
\end{lemma}

\begin{proof}
Let $\epsilon > 0$, and let $x \in \gamma$ and $y \in \gamma'$ be such that $\gamma$ and $\gamma'$ reach $x$ and $y$ after reaching $x'$ and $y'$, respectively. By the definition of the Gromov product, we can choose $x, y$ sufficiently far along $\gamma, \gamma'$ such that $(x \mid y)_o \geq (\xi \mid \eta)_o - \epsilon$.

Since $\gamma, \gamma'$ are $D$-quasi-rulers, 
\[ d(o, x') - D \leq ( x' \mid x)_o \leq d(o, x'), \] and
\[ d(o, y') - D \leq ( y' \mid y)_o \leq d(o, y'). \]
Using $\delta$-hyperbolicity of $X$, we estimate
\begin{equation*}
    \begin{split}
        (x' \mid y')_o & \geq \min\{ (x' \mid x)_o, (x \mid y)_o, (y \mid y')_o \} - 2\delta\\
        & \geq \min\{ d(o, x') - D, (\xi \mid \eta)_o - \epsilon, d(o, y') - D \} - 2\delta\\
        & \geq \min\{ d(o,x'), d(o,y') \} - D - \epsilon - 2\delta.
    \end{split}
\end{equation*}
Since we can obtain this inequality for all $\epsilon$, we conclude that
\[ (x' \mid y')_o \geq \min\{ d(o,x'), d(o,y') \} - D - 2\delta, \]
and therefore,
\begin{equation*}
    \begin{split}
        d(x', y') & = d(o,x') + d(o,y') - 2 (x' \mid y')_o\\
        & \leq d(o,x') + d(o,y') - 2 \min\{ d(o,x'), d(o,y') \}  + D + 2 \delta\\
        & = \vert d(o,x') - d(o,y') \vert + D + 2 \delta.\qedhere
    \end{split}
\end{equation*}
\end{proof}

\bibliography{mybib}
\bibliographystyle{alpha}

\end{document}